\documentclass[a4paper]{amsart}

\usepackage{amsthm, amssymb, amsmath, amsfonts, url}

\newtheorem{thm}{Theorem}[section]
\newtheorem{prop}[thm]{Proposition}
\newtheorem{lem}[thm]{Lemma}
\newtheorem{cor}[thm]{Corollary}

\renewcommand{\le}{\leqslant}
\renewcommand{\ge}{\geqslant}

\newcommand{\E}{\mathbb{E}}
\newcommand{\EE}{\mathbf{E}}
\newcommand{\EEo}{\mathbf{E}^\omega}
\newcommand{\cE}{\mathcal{E}}
\newcommand{\N}{\mathbb{N}}
\renewcommand{\L}{\mathcal{L}}
\newcommand{\1}{\mathbf{1}}
\newcommand{\R}{\mathbb{R}}
\newcommand{\Z}{\mathbb{Z}}
\renewcommand{\P}{\mathbb{P}}
\newcommand{\PP}{\mathbf{P}}
\newcommand{\PPo}{\mathbf{P}^\omega}
\newcommand{\ov}{\overline}
\newcommand{\td}{\tilde}
\newcommand{\eps}{\varepsilon}
\def\d{{\mathrm{d}}}
\newcommand{\cro}{[\omega,\sigma]}
\newcommand{\tri}{| \! | \! |}
\newcommand{\grad}{\nabla}

\newcommand{\var}{\mathbb{V}\mathrm{ar}}
\newcommand{\rN}{\mathrm{N}}
\newcommand{\mclN}{\mathcal{N}}

\title[Variance decay for the environment viewed by the particle]{Variance decay for functionals of the environment viewed by the particle}
\author{Jean-Christophe Mourrat}

\address{Universit\'e de Provence, CMI, 39 rue Joliot Curie, 13013 Marseille, France ; PUC de Chile, Facultad de Matem\'aticas, Vicu\~na Mackenna 4860, Macul, Santiago, Chile.}

\begin{document}
\begin{abstract}
For the random walk among random conductances, we prove that the environment viewed by the particle converges to equilibrium polynomially fast in the variance sense, our main hypothesis being that the conductances are bounded away from zero. The basis of our method is the establishment of a Nash inequality, followed either by a comparison with the simple random walk or by a more direct analysis based on a martingale decomposition. As an example of application, we show that under certain conditions, our results imply an estimate of the speed of convergence of the mean square displacement of the walk towards its limit.

\bigskip
\noindent \textsc{R\'esum\'e.} Pour la marche al\'eatoire en conductances al\'eatoires, nous montrons 
que l'environnement vu par la particule converge vers l'\'equilibre \`a une vitesse polynomiale au sens de la variance, notre hypoth\`ese principale \'etant que les conductances sont uniform\'ement minor\'ees. Notre m\'ethode se base sur l'\'etablissement d'une in\'egalit\'e de Nash, suivie soit d'une comparaison avec la marche al\'eatoire simple, soit d'une analyse plus directe fond\'ee sur une m\'ethode de martingale. Comme exemple d'application, nous montrons que sous certaines conditions, ces r\'esultats permettent d'estimer la vitesse de convergence vers sa limite du d\'eplacement quadratique moyen de la marche.



\end{abstract}
\maketitle
\section{Introduction}
When considering some large scale property of a heterogeneous environment, it is natural to expect that the local fluctuations average out, and that one can replace the irregular medium by an ``averaged'' one, described by a small number of effective parameters. This problem of homogenization of heterogeneous media is old, and can be traced back at least to \cite{maxwell} and \cite{rayleigh}. Mathematical results concerning the homogenization of periodic environments began to appear around 1970 (see for instance \cite[Chapter 1]{jko}, and references therein). The study of averaging of random environments started with the works of \cite{kozlov1}, \cite{yuri1}, and \cite{papavara1}, where stochastic homogenization was obtained for divergence form elliptic operators. These analytic results have their probabilistic counterpart, in terms of invariance principles for certain diffusions in random environment \cite{kun,osada}.

A central question follows any homogenization result~: when can one replace, up to some given precision, the heterogeneous medium by the averaged one~?

As discussed in \cite[p.~199-205]{molchanov}, and contrary to the periodic case (see for instance \cite[p.~151-152]{molchanov} and \cite[Section 2.6]{jko}), the typical space scale of the averaging of a random environment may be unexpectedly large. As a matter of fact, very little is known about this issue. A notable exception is \cite{yuri}, where the author considers the Poisson equation on a bounded domain of $\R^d$, for divergence form elliptic operators. It is shown that the solution corresponding to a typical length scale of order $\eps$ converges, as $\eps$ tends to $0$, to the solution of the averaged problem faster than some power of $\eps$, where the exponent depends on the dimension $d \ge 3$, the ellipticity constant and some mixing condition.

The related problem of finding an efficient way to compute the effective parameters of the averaged environment is also troublesome. Let us consider a random walk or a diffusion in random environment, and assuming it exists, let us write $D$ for the effective diffusion matrix. As the periodic case is better understood, it is natural to consider periodizations of the initial random  environment $\omega$. For periods of size $n$, this procedure defines an effective diffusion matrix $D_n(\omega)$, that might be a good approximation of $D$. Considering divergence form elliptic operators and using results from \cite{yuri}, it is shown in \cite{bourpia} that $D_n(\omega)$ converges to $D$ faster than $C n^{-\alpha}$, where $\alpha$ depends on the dimension $d \ge 3$, the ellipticity constant and some mixing condition. In the case of a random walk among random independent conductances, \cite{capuio} and \cite{boivin} have shown that, under an ellipticity condition, $D_n(\omega)$ converges to $D$ almost surely (but without quantitative estimates on this convergence), and provide estimates of the variance of $D_n(\omega)$ for fixed $n$.

Many proofs of homogenization results rely on the ergodicity of an auxiliary process introduced in \cite{papavara2}, that is now usually called \emph{the environment viewed by the particle} (see for instance \cite{osada}, \cite{kozlov2}, \cite{kipvar} or \cite{masi}). Naturally, the ergodic theorem gives only an asymptotic result. Our main purpose here is to provide, in the context of random walks among random conductances, an estimate of the speed of convergence to equilibrium of the environment viewed by the particle, our central assumption being that the conductances are bounded away from zero. We obtain a polynomial decay of the variance of a large class of functionals. Under the additional hypothesis that the conductances are also bounded from above and when $d \ge 7$, we can derive information on the rate of convergence of the mean square displacement of the walk towards its limit.

We would like to draw the reader's attention to the fact that, in the aforementioned papers, the proofs of algebraic speed of convergence rely on analytical tools such as Harnack's inequality. As a consequence, the exponent found in the polynomial decay is kept implicit, and depends on the ellipticity constant. In contrast, the exponent we find here is given explicitly in terms of the dimension only, and the polynomial decay holds for possibly non-elliptic environments.

We now define our present setting with more precision. Consider on $\Z^d$ ($d \ge 1$) the nearest neighbour relation : $x,y \in \Z^d$ are neighbours (written $x \sim y$) if and only if $\|x-y\|_2 = 1$. Drawing an (unoriented) edge between any two neighbours turns $\Z^d$ into a graph, and we write $\mathbb{B}^d$ for the set of edges. We define a \emph{random walk among random conductances} on $\Z^d$ as follows. 

Let $\Omega = (0,+\infty)^{\mathbb{B}^d}$ ; we call an element $\omega = (\omega_e)_{e \in \mathbb{B}^d} \in \Omega$ an \emph{environment}. If $e = (x,y) \in \mathbb{B}^d$, we may write $\omega_{x,y}$ instead of $\omega_e$. By construction, $\omega$ is symmetric~: $\omega_{x,y} = \omega_{y,x}$.

For any $\omega \in \Omega$, we consider the Markov process $(X_t)_{t \ge 0}$ with jump rate between $x$ and $y$ given by $\omega_{x,y}$. We write $\PPo_x$ for the law of this process starting from $x \in \Z^d$, $\EEo_x$ for its associated expectation.

The environment $\omega$ will be itself a random variable, whose law we write $\P$ (and $\E$ for the corresponding expectation). There are translation operators $(\theta_x)_{x \in \Z^d}$ acting on $\Omega$, given by $(\theta_x \ \omega)_{y,z} = \omega_{x+y,x+z}$. We assume the measure $\P$ to be invariant under these translations. Moreover, we assume, without further mention, that the random walk is well defined for all times, or in other words, that it does not travel along its trajectory in finite time. This assumption is satisfied whenever one can find a threshold such that the set of conductances above this threshold does not percolate. In particular, it holds if $\P$ is a product measure.

We define the averaged (or \emph{annealed}) law as
$$
\ov{\P}[\cdot] = \P\big[\PPo_0[\cdot]\big],
$$
and $\ov{\E} = \E.\EEo_0$ for the corresponding expectation. 

Assuming that the conductances are integrable, and that $\P$ is ergodic (with respect to $(\theta_x)_{x \in \Z^d}$), \cite{masi} have shown that the process $(\eps X_{\eps^{-2} t})_{t \ge 0}$ converges to a Brown\-ian motion  under the annealed law, as $\eps$ goes to $0$. The proof of this result uses the ergodicity of \emph{the environment viewed by the particle}. This process is defined by $\omega(t) = \theta_{X_t} \ \omega$. It is a Markov process for which the measure $\P$ is reversible, and whose infinitesimal generator is given by
$$
\L f (\omega) = \sum_{|z| = 1} \omega_{0,z} (f(\theta_z \ \omega) - f(\omega)).
$$

The aim of this note is to provide a quantitative information about the speed of convergence of $(\omega(t))$ towards equilibrium, our central assumption being that for any $e \in \mathbb{B}^d$, $\omega_e \ge 1$. More precisely, let us define $f_t(\omega) = \EEo_0[f(\omega(t))]$. We would like to find appropriate functional $V$ and exponent $\alpha > 0$ such that, for any function $f : \Omega \to \R$ satisfying $\E[f] = 0$~:
\begin{equation}
\label{obj}
\E[(f_t)^2] \le \frac{V(f)}{t^{\alpha}},
\end{equation}
with $V(f)$ finite at least for functions that are bounded and depend only on a finite number of coordinates.

A first attempt to solve the problem could be to decompose $\EEo_0[f(\omega(t))]$ into
$$
\sum_{x \in \Z^d} f(\theta_x \ \omega) \PPo_0[X_t = x],
$$
and look for Gaussian bounds on the transition probabilities. Although this can be an efficient strategy when the random walk is not influenced by the environment (and indicates that $\alpha$ should be equal to $d/2$), it seems bound to fail in our context, as the random variables $f(\theta_x \ \omega)$ and $\PPo_0[X_t = x]$ are correlated in a rather problematic way.

In the not so distant context of interacting particle systems, first results concerning polynomial convergence to equilibrium were obtained in \cite{ligg,deuschel}. We will here rely on the general technique proposed in \cite[Theorem 2.2]{ligg}, where it is shown that, under some assumption on the functional $V$ involved in (\ref{obj}), the polynomial convergence to equilibrium of the process is equivalent to a certain Nash inequality.

In our context, we derive a Nash inequality from the knowledge of the spectral gap for the dynamics restricted to a finite box, in a way that is similar to the one contained in \cite{bert-zeg} (section~\ref{s:specgap}). 

An important issue is the choice of the functional $V$ in equation~(\ref{obj}). As one can see in \cite[Theorem 2.2]{ligg}, a desirable feature for $V$ is to be contractive under the action of the semi-group. For the case of interacting particle systems, this contractivity property is usually obtained using some monotonicity of the model considered. In our present context, there is no such property at our disposal, and we are left with a functional that turns out not to be contractive.

We propose two different approaches to overcome this difficulty. The first is to consider instead the particular case when the law of the random walk does not depend on the environment (section~\ref{s:simpleRW}). Here, the contractivity property holds, and one thus easily obtains algebraic convergence to equilibrium, with exponent $\alpha = d/2$. Using a comparison of resolvents between the simple random walk and the random walk in random environment (section~\ref{s:resolvents}), one can partially transfer this result back to the original random walk among random conductances, obtaining (if one forgets about logarithmic corrections) an exponent $\alpha = \min(d/2,1)$.

This exponent is rather unsatisfactory, especially when the dimension is large. We therefore provide a second method in section~\ref{s:martmeth} to circumvent the absence of contractivity, based on a martingale method. This enables us to obtain an algebraic decay of the form~(\ref{obj}) with a new functional and the exponent $\alpha = d/2 - 2$. This improves the previous result as soon as $d \ge 7$, although it requires a more restrictive condition on $f$, implying in particular that it is bounded. 

One might argue that systems of particles and random walks in random environments are similar problems. Indeed, they look close to one another, since a tagged particle in a system of interacting particles can be seen as a random walk in a random environment. However, if one considers the environment seen by this tagged particle, it constantly changes with time, even when the particle does not move. On the other hand, the environment seen by the random walk among random conductances is static, and evolves only via translations, making the convergence to equilibrium more difficult.

We now turn our attention to the consequences of our results. Observe that, as $\P$ is reversible for $(\omega(t))_{t \ge 0}$, the associated semi-group is self-adjoint in $L^2(\P)$. Hence, it comes that
\begin{equation}
\label{correlations}
\E[(f_t)^2] = \E[f(\omega) f_{2t}(\omega)] = \ov{\E}[f(\omega(0))f(\omega(2t))] .
\end{equation}
As a consequence, our problem is in fact equivalent to a control of the decay of the correlations of $(\omega(t))_{t \ge 0}$. It is shown in \cite{kipvar} that, provided these correlations are integrable, an annealed invariance principle holds for 
\begin{equation}
\label{defZt}
Z_t = \int_0^t f(\omega(s)) \d s .
\end{equation}
In fact, if the correlations in (\ref{correlations}) decay faster than $t^{-\alpha}$ for some $\alpha > 1$, we will see that one can control the speed of convergence of $\ov{\E}[(Z_t)^2]/t$ towards its limit. This kind of result can be seen as a (rather weak) quantitative annealed central limit theorem. As an example, when $d \ge 7$ and the conductances are bounded from above, we can estimate the speed of convergence of the mean square displacement $\ov{\E}[(\|X_t\|_2)^2]/t$ towards its limit.

As noted before, because of its relative simplicity and of the possibility to relate it with our initial problem, we will also consider the simple random walk $X^\circ_t$ (for which the jump rates are uniformly equal to $1$), together with $\omega^\circ(t) = \theta_{X^\circ_t} \ \omega$. This last process is also reversible with respect to $\P$, and has for infinitesimal generator
$$
\L^\circ f (\omega) = \sum_{|z| = 1} (f(\theta_z \ \omega) - f(\omega)).
$$
We will write $f^\circ_t(\omega)$ for $\EEo_0[f(\omega^\circ(t))]$.

\noindent \textbf{Addendum.} A referee pointed out to us the references \cite{cudna, gloria1, gloria2}, that we were not previously aware of. In these papers, the problem of estimating the effective diffusion matrix $D$ is investigated. This matrix can be expressed in terms of a so-called ``corrector field'', that solves a Poisson equation. Following \cite{yuri}, \cite{cudna, gloria1,gloria2} propose to approach this corrector by the practically computable solution of a regularized Poisson equation. In \cite{cudna}, the authors consider the case when the conductances can take only two different values, and give an explicit bound between the diffusion matrix and its approximation. Most interestingly, in \cite{gloria1, gloria2}, the authors tackle the problem of general bounded conductances. They decompose the difference between the matrix $D$ and its approximation into the sum of two error terms, one being due to the discrepancy between $D$ and the expected value of the approximation, and the other to the random fluctuations of the approximation. They provide the precise asymptotic behaviour of this second term in \cite{gloria1}, and of the first one in \cite{gloria2}. There is an interesting interplay between their approach and ours, which we will describe in the last section of this paper. 

%
%
%
%
%
%
\section{Statement of the main results}
\label{s:statement}
\setcounter{equation}{0}
From now on, we always assume that the conductances are uniformly bounded from below by $1$ : for any $e \in \mathbb{B}^d$, we impose that $\omega_e \ge 1$.
We define the following possible additional hypothesis on the regularity of $\P$~:

\noindent \textbf{Assumption $(\mathfrak{I})$ :} The random variables $(\omega_e)_{e \in \mathbb{B}^d}$ are independent and identically distributed.

\noindent \textbf{Assumption $(\mathfrak{A})$ :} Assumption $(\mathfrak{I})$ is satisfied, and the conductances are uniformly bounded from above.

To state our main results, we need to fix some notations. We write $B_n$ for $\{-n,\ldots,n\}^d$. For a function $f : \Omega \to \R$, let
\begin{equation}
\label{defSn}
S_n(f) = \sum_{x \in B_n} f(\theta_x \ \omega),
\end{equation}
and
\begin{equation}
\label{defmclN}
\mathcal{N}(f) = \sup_{n \in \N} \frac{1}{|B_n|} \E\left[\left(S_n(f) \right)^2 \right].
\end{equation}
Note that, using the translation invariance of $\P$, we have for $f \in L^2(\P)$~:
\begin{equation*}
\frac{1}{|B_n|} \E\left[\left(S_n(f) \right)^2 \right] = \frac{1}{|B_n|} \sum_{x,y \in B_n} \E[f(\theta_x \ \omega) f(\theta_y \ \omega)]  \le  \sum_{x \in \Z^d} |\E[f(\omega) f(\theta_x \ \omega)]| .
\end{equation*}
In particular, $\mathcal{N}(f)$ is finite under assumption ($\mathfrak{I}$) if $f \in L^2(\P)$ satisfies $\E[f] = 0$ and depends only on a finite number of coordinates. As a rule of thumb, the reader is advised to remember that the fact that $\mathcal{N}(f)$ is finite should imply that $\E[f] = 0$ (and our results, together with Birkhoff's ergodic theorem, show that it is indeed the case).

For each edge $e \in \mathbb{B}^d$, let~:
\begin{equation}
\label{def:grad}
| \grad f | (e) = \sup | f(\omega) - f(\omega')|,
\end{equation}
where the sup is taken over all $\omega, \omega'$ in the support of $\P$ such that $\omega = \omega'$ except possibly on $e$. We define the semi-norm~:
\begin{equation}
\label{def:tri}
\tri f \tri = \sum_{e \in \mathbb{B}^d} | \grad f | (e),
\end{equation}
and $\rN(f) = \tri f \tri^2 + \|f\|_\infty^2$. For example, $\rN(f)$ is finite if $f$ is a bounded function that depends only on a finite number of coordinates.
We write $\ln_+(t)$ for $\max(1,\ln(t))$.

We first obtain a decay of the variance of $f^\circ_t$ for certain functions $f$.
\begin{thm}
\label{main:indep}
There exists $C > 0$ such that for any $f : \Omega \to \R$, if $\mathcal{N}(f)$ is finite, then for any $t > 0$, we have~:
$$
\E[(f^\circ_t)^2] \le C \ \frac{\mclN(f)}{t^{d/2}}.
$$
\end{thm}

This result has its partial counterpart concerning the random walk among random conductances, as follows.

\begin{thm}
\label{main:d<}
There exists $C > 0$ such that for any $f : \Omega \to \R$, if $\mathcal{N}(f)$ is finite, then for any $t > 0$, we have~:
\begin{eqnarray*}
\E[(f_t)^2] \le C \ \frac{\mclN(f)}{\sqrt{t}}  & \text{ if } d = 1, \\
\E[(f_t)^2] \le C \ \frac{\mclN(f)\ln_+(t)}{t}  & \text{ if } d = 2, \\
\E[(f_t)^2] \le C \ \frac{\mclN(f)}{t}  & \text{ if } d \ge 3 .
\end{eqnarray*}
Moreover, if $d \ge 3$, we also have~:
\begin{equation}
\label{main:condkipvar}
\int_0^{+\infty} \E[(f_t)^2] \ \d t \le C \mclN(f).
\end{equation}
\end{thm}

\noindent \textbf{Remarks.} As the proof reveals, the constant $C$ appearing in Theorems~\ref{main:indep} and \ref{main:d<} can be chosen so that the results hold for any law $\P$ (provided it satisfies the always assumed uniform lower bound $\omega \ge 1$). Moreover, one can generalize these results to cases when $\mclN(f)$ is infinite, see Proposition~\ref{generalization}.

For larger dimensions, we obtain the following result.

\begin{thm}
\label{main:d>}
Under assumption $(\mathfrak{I})$, there exists $C$ such that for any $f : \Omega \to \R$, if $\rN(f)$ is finite and $\E[f] = 0$, then for any $t > 0$, we have~:
\begin{eqnarray*}
\E[(f_t)^2] \le C \ \frac{\rN(f)}{t^{d/2-2} } & \text{ if } d \ge 5.
\end{eqnarray*}
\end{thm}

\noindent \textbf{Remarks.} If $\eta$ is some positive real number, one can choose $C$ so that the result in Theorem~\ref{main:d>} is valid uniformly for any law $\P$ satisfying the property given in equation~(\ref{q:percolepas}). It turns out however that Theorem~\ref{main:d<} is stronger when $d \le 6$. Indeed, one can show using the martingale method introduced in section~\ref{s:martmeth} that, under assumption ($\mathfrak{I}$), we have $\mclN(f) \le \rN(f)$. Finally, we point out that, for a specific class of functionals, Proposition~\ref{gtraj} provides a variation on this result that is also of interest.

Our starting point (section~\ref{s:specgap}) is the existence of spectral gap inequalities for the dynamics restricted to a finite box of size $n$. For this dynamics, it gives the rate of convergence of the environment viewed by the particle towards the empirical measure (that is, the uniform measure over any translation of the environment, provided we keep the range of the translations inside the box). We do some computation that makes the measure $\P$ come into play, and a corrective term appears (see Proposition~\ref{seminash}). 

In the case of $(\omega^\circ(t))$, we show in section~\ref{s:simpleRW} that this corrective term decays with the time evolution, thus leading to Theorem~\ref{main:indep}. Noting that the Dirichlet form associated with $(\omega^\circ(t))$ is dominated by the one of $(\omega(t))$, we obtain Theorem~\ref{main:d<} in section~\ref{s:resolvents} by a comparison of the resolvents of these processes. 

A second approach is to estimate directly this (not necessarily decreasing with time) corrective term for the random walk under random conductances. We do so under assumption ($\mathfrak{I}$) in section~\ref{s:martmeth} via a martingale method. We finally prove Theorem~\ref{main:d>} (together with some variations) in section~\ref{s:concl}, using the former results. 

We now present some consequences of the previous theorems. 

Whenever $f \in L^2(\P)$, we define $e_f$ as the spectral measure of $-\L$ (as a self-adjoint operator on $L^2(\P)$) projected on the function $f$, so that for any bounded continuous $\Psi : [0,+\infty) \to \R$~:
\begin{equation}
\label{defef}
\E[\Psi(-\L)(f)(\omega) f(\omega)] = \int \Psi(\lambda) \d e_f(\lambda).
\end{equation}

Lemma~\ref{lemresolvent} states that~:
\begin{equation}
\label{condkipvar}
\int_0^{+\infty} \E[(f_{t})^2] \ \d t = 2 \int \frac{1}{\lambda} \ \d e_f(\lambda).
\end{equation}
Hence, a consequence of equation~(\ref{main:condkipvar}) is that, whenever $d \ge 3$ and $\mclN(f)$ is finite,
$$
\int \frac{1}{\lambda} \ \d e_f(\lambda) < + \infty.
$$

This condition is shown in \cite{kipvar} to be necessary and sufficient to ensure that $(Z_t)_{t \ge 0}$ (defined in~(\ref{defZt})) satisfies an invariance principle. More precisely, the authors show that there exist $(M_t)_{t \ge 0}$, $(\xi_t)_{t \ge 0}$ such that $Z_t = M_t + \xi_t$, where $(M_t)$ is a martingale with stationary increments under $\ov{\P}$ (and the natural filtration), and $(\xi_t)$ is such that~:
\begin{equation}
\label{xitendvers0}
\frac{1}{t} \ov{\E}[(\xi_t)^2] \xrightarrow[t \to + \infty]{} 0.
\end{equation}
Using this decomposition, they prove that $(\eps Z_{t/\eps^2})_{t \ge 0}$ converges, as $\eps$ goes to $0$, to a Brownian motion of variance
\begin{equation}
\label{defsigma}
\sigma^2 = \ov{\E}[(M_1)^2] = 2 \int \frac{1}{\lambda} \ \d e_f(\lambda).
\end{equation}

We will show first that an algebraic decay (with exponent strictly greater than~$1$) of the variance of $(f_t)$ is equivalent to a particular behaviour of the spectral measure $e_f$ close to $0$. 

\begin{thm}
\label{tcl}
Let $\alpha > 1$ and $f \in L^2(\P)$.
The following statements are equivalent~:
\begin{enumerate}
\item
\label{tcl:equiv1}
There exists $C > 0$ such that for any $t > 0$,
$$
\E[(f_t)^2] \le \frac{C}{t^\alpha}.
$$
\item
\label{tcl:equiv3}
There exists $C > 0$ such that for any $\delta > 0$,
$$
\int_{[0,\delta]} \frac{1}{\lambda} \ \d e_f(\lambda) \le C \delta^{\alpha - 1}.
$$
\end{enumerate}
\end{thm}

It turns out that this additional control of the spectral measure enables us to estimate the speed of convergence in equation~(\ref{xitendvers0}). This provides us enough information to estimate the speed of convergence of $\ov{\E}[(Z_t)^2]/t$ towards its limit.

For any $\alpha > 1$, let $\psi_\alpha : [0,+\infty) \to \R$ be defined by
\begin{equation}
\label{defpsi}
\psi_\alpha(t) = 
\left|
\begin{array}{ll}
t^{\alpha-1} & \text{if } \alpha < 2 \\
t/(\ln_+(t)) & \text{if } \alpha = 2 \\
t & \text{if } \alpha > 2 .
\end{array}
\right. 
\end{equation}

\begin{thm}
\label{xipetit}
Under one of the equivalent conditions (\ref{tcl:equiv1}), (\ref{tcl:equiv3}) of Theorem~\ref{tcl} (and in particular under assumption $(\mathfrak{I})$, if $\rN(f)$ is finite and $d \ge 7$) : 
\begin{enumerate}
\item
One can construct $(M_t)$, $(\xi_t)$ such that $Z_t = M_t + \xi_t$, where $(M_t)$ is a martingale with stationary increments under $\ov{\P}$, and $(\xi_t)$ is such that, for some $C > 0$~:
$$
\frac{1}{t} \ov{\E}[(\xi_t)^2] \le \frac{C}{\psi_\alpha(t)}.
$$
\item
There exists $C > 0$ such that for any $t \ge 0$ :
\begin{equation*}
\label{tclspeedgen}
0 \le \sigma^2 - \frac{1}{t} \ov{\E}\left[\left( Z_t \right)^2\right] \le \frac{C}{\psi_\alpha(t)},
\end{equation*}
where $\sigma$ is defined by (\ref{defsigma}).
\end{enumerate}
\end{thm}

In \cite[Section 4]{kipvar} and \cite[Section 4]{masi}, the authors prove in particular that under assumption $(\mathfrak{A})$, the random walk satisfies an invariance principle : $(\eps X_{t/\eps^2})_{t \ge 0}$ converges, as $\eps$ goes to $0$, to a Brownian motion of covariance matrix $\ov{\sigma}^2 \rm{Id}$. A consequence of Theorems~\ref{main:d>} and \ref{xipetit} is that, when $d \ge 7$, we can give an estimate of the speed of convergence of the mean square displacement towards its limit.

\begin{cor}
\label{tclcor}
Under assumption $(\mathfrak{A})$ and when $d \ge 7$, there exists $C > 0$ such that for any $t \ge 0$~:
\begin{equation}
\label{tclspeedex}
0 \le \frac{1}{t} \ov{\E}\left[(\| X_t \|_2)^2\right] - d \ov{\sigma}^2 \le \frac{C}{\psi_{\alpha}(t)} \qquad \text{ with } \alpha = \frac{d}{2}-2.
\end{equation}
\end{cor}

Theorems \ref{tcl} and \ref{xipetit} are proved in section~\ref{s:clt}, together with Corollary~\ref{tclcor}.

\noindent \textbf{Remark.} As we discuss in the addendum, results from \cite{gloria2} enable to strengthen Corollary~\ref{tclcor} when $d \le 8$ (see in particular Corollary~\ref{cor:gloria}).

\noindent \textbf{Some more notations.} We introduce the Dirichlet forms associated with the processes $(\omega(t))$ and $(\omega^\circ(t))$, respectively~:
\begin{equation*}
\cE(f,f) = - \E[\L f (\omega) f (\omega)] = \frac{1}{2} \sum_{|z| = 1} \E\left[\omega_{0,z}(f(\theta_z \ \omega) - f(\omega))^2\right],
\end{equation*}
and
\begin{equation*}
\cE^\circ(f,f) = - \E[\L^\circ f (\omega) f (\omega)] = \frac{1}{2} \sum_{|z| = 1} \E\left[(f(\theta_z \ \omega) - f(\omega))^2\right].
\end{equation*}
For $A$ a subset of $\Z^d$, we will write $|A|$ for its cardinal. We also define its inner boundary as
$$
\underline{\partial}A = \{ x \in A : \exists y \in \Z^d \setminus A \ x \sim y \}.
$$
The letter $C$ refers to a strictly positive number, that may not be the same from one occurrence to another. 

%
%
%
%
%
%
\section{From spectral gap to Nash inequality}
\label{s:specgap}
\setcounter{equation}{0}

\begin{prop}[Spectral gap]
\label{specgap}
There exists $C_S > 0$ such that for any $n \in \N$ and any function $g : B_n \to \R$, we have~:
$$
\sum_{x \in B_n} (g(x) - m_n(g))^2 \le \frac{C_S}{4} n^2 \sum_{\substack{x,y \in B_n \\ x \sim y}} (g(y)-g(x))^2,
$$
where $m_n(g)$ is given by
$$
m_n(g) = \frac{1}{|B_n|} \sum_{x \in B_n} g(x).
$$
\end{prop}
\begin{proof}
A lower bound on the isoperimetric constant is given by \cite[Theorem 3.3.9]{SC}, which implies a lower bound on the spectral gap via Cheeger's inequality (\cite[Lemma 3.3.7]{SC}). 
\end{proof}
\begin{prop}
\label{seminash}
For any function $f \in L^2(\P)$ and any $n \in \N$, we have :
$$
\E[f(\omega)^2] \le C_S n^2 \cE(f,f) + \frac{2}{|B_n|^2} \E[S_n(f)^2],
$$
where $S_n(f)$ was defined in (\ref{defSn}).
Moreover, the same inequality holds with $\cE$ replaced by $\cE^\circ$.
\end{prop}
\begin{proof}
As we assumed that $\omega \ge 1$, it is clear that $\cE(f,f) \ge \cE^\circ(f,f)$, so it is enough to show the claim for $\cE^\circ$.

Let $g(x) = f(\theta_x \ \omega)$.
Using the translation invariance of $\P$ and the fact that $(a+b)^2 \le 2a^2 + 2b^2$, we have~:
$$
\E[f^2] = \E[g(x)^2] \le 2\E[(g(x) - m_n(g))^2] + 2\E[m_n(g)^2].
$$
Summing this inequality over all $x \in B_n$, and using Proposition \ref{specgap}, we obtain :
$$
|B_n| \E[f^2] \le \frac{C_S}{2} n^2 \sum_{\substack{x,y \in B_n \\ x \sim y}} \E[(f(\theta_y \ \omega)-f(\theta_x \ \omega))^2] + 2 |B_n|\E[m_n(g)^2].
$$
Note that 
$$
\sum_{\substack{x,y \in B_n \\ x \sim y}} \E[(f(\theta_y \ \omega)-f(\theta_x \ \omega))^2] \le \sum_{x \in B_n} \sum_{\substack{y \in \Z^d \\ y \sim x}} \E[(f(\theta_y \ \omega)-f(\theta_x \ \omega))^2].
$$
As $\P$ is invariant under translation, the sum
$$
\sum_{\substack{y \in \Z^d \\ y \sim x}} \E[(f(\theta_y \ \omega)-f(\theta_x \ \omega))^2]
$$ 
is in fact independent of $x$, and it comes that :
$$
|B_n| \E[f^2] \le \frac{C_S}{2} n^2 |B_n| \sum_{|z| = 1} \E\left[(f(\theta_z \ \omega) - f(\omega))^2\right] + 2 |B_n|\E[m_n(g)^2].
$$
We get the result dividing by $|B_n|$ and noting that $m_n(g) = S_n(f)/|B_n|$.
\end{proof}
We need to modify slightly $\mclN(f_t)$, to ensure that it does not become too small when $t$ goes to infinity. We define $\mclN'(f_t) = \max(\mclN(f_t),\|f\|_2^2)$.
\begin{prop}[Nash inequality]
\label{nash}
There exists $C > 0$ such that for any $f \in L^2(\P)$ and any $t \ge 1$~:
$$
\E[(f_t)^2] \le C \cE(f_t,f_t)^{d/(d+2)} \mclN'(f_t)^{2/(d+2)},
$$
and the same inequality holds with $f_t$, $\cE$ replaced by $f^\circ_t$, $\cE^\circ$.
\end{prop}
We first prove the following classical result.
\begin{lem}
\label{decroissdirichl}
For any $f \in L^2(\P)$ and any $t > 0$, we have
\begin{equation*}
\cE^\circ(f_t,f_t) \le \cE(f_t,f_t) \le \frac{1}{2 e t} \|f\|_2^2.
\end{equation*}
\end{lem}
\begin{proof}
As was already mentioned, the first inequality is clear due to the assumption that $\omega \ge 1$. 

The Dirichlet form $\cE(f_t,f_t)$ can be expressed in terms of the spectral measure $e_f$ as (see (\ref{defef}) for the definition of $e_f$)~:
$$
\cE(f_t,f_t) = \int \lambda e^{-2 \lambda t} \d e_f(\lambda).
$$
Noting that for any $x \ge 0$, we have $x e^{-x} \le 1/e$, we obtain that $\lambda e^{-2 \lambda t} \le 1/(2 e t)$, and the result follows. 
\end{proof}
\begin{proof}[Proof of Proposition~\ref{nash}]
Using the fact that $|B_n| \ge 2 n^d$, together with the definition of $\mclN(f)$ given in (\ref{defmclN}), Proposition \ref{seminash} gives
\begin{equation}
\label{presqnash}
\E[(f_t)^2] \le C_S n^2 \cE(f_t,f_t) + n^{-d} \mclN'(f_t).
\end{equation}
Let $w$ be the positive real number satisfying
$$
w^{d+2} = \frac{\mclN'(f_t)}{2 e \cE(f_t,f_t)}.
$$
Using the fact that $\mclN'(f_t) \ge \|f\|_2^2$, we know by Lemma~\ref{decroissdirichl} that $w \ge 1$ whenever $t \ge 1$. In particular, the integer part $\lfloor w \rfloor$ satisfies $w/2 \le \lfloor w \rfloor \le w$. Taking $n = \lfloor w \rfloor$ in equation (\ref{presqnash}) (and $t \ge 1$), we obtain
$$
\E[(f_t)^2] \le C \cE(f_t,f_t)^{d/(d+2)} \mclN'(f_t)^{2/(d+2)},
$$
with
$$
C = C_S (2e)^{-2/(d+2)} + 2^d (2e)^{d/(d+2)}.
$$
\end{proof}
\begin{prop}
\label{decaygeneral}
There exists $C > 0$ such that for any $f \in L^2(\P)$ and any $t \ge 1$~:
$$
\E[f_t^2] \le C\left( \int_1^t \mclN'(f_s)^{-2/d} \ \d s \right)^{-d/2},
$$
and the same inequality holds with $f_t, f_s$ replaced by $f^\circ_t, f^\circ_s$.
\end{prop}
\begin{proof}
Noting that $\partial_t \E[f_t^2] = - 2 \cE(f_t,f_t)$, the inequality obtained in Proposition~\ref{nash} becomes a differential inequality~:
$$
\cE(f_t,f_t) (\E[f_t^2])^{-(1+2/d)} = \frac{d}{4} \partial_t \big[(\E[f_t^2])^{-2/d}\big]  \ge C^{-(1+2/d)} \mclN'(f_t)^{-2/d}.
$$
Integrating this inequality, we are led to~:
$$
(\E[f_t^2])^{-2/d} \ge \frac{4}{d} C^{-(1+2/d)} \int_1^t \mclN'(f_s)^{-2/d} \ \d s ,
$$
which shows the proposition for $f_t$. The same proof applies to $f^\circ_t$ as well.
\end{proof}

%
%
%
%
\section{Variance decay for the simple random walk}
\label{s:simpleRW}
\setcounter{equation}{0}
\begin{prop}
\label{contract}
For any integer $n$ and any $f \in L^2(\P)$, the function
$$
t \mapsto \E[(S_n(f^\circ_t))^2]
$$
is decreasing. 
On the other hand, there exist a law $\P$ satisfying $(\mathfrak{A})$, a bounded function $f : \Omega \to \R$ that depends on a finite number of coordinates, and an integer $n$ such that the function
$$
t \mapsto \E[(S_n(f_t))^2]
$$
is not decreasing.
\end{prop}
\begin{proof}
When the walk is independent from the environment, environment translations and time evolution commute, in the sense that~:
\begin{equation}
\label{commute}
\E[S_n(f^\circ_t)^2] = \E\left[\big((S_n(f))^\circ_t\big)^2\right],
\end{equation}
and as a consequence~:
$$
\partial_t \E[S_n(f^\circ_t)^2] = - 2 \cE_0((S_n(f))^\circ_t,(S_n(f))^\circ_t) \le 0.
$$

The commutation property (\ref{commute}) no longer holds for the random walk on random conductances $(X_t)$. We construct an example showing that $t \to \E[S_n(f_t)^2]$ may increase. We fix $d = 1$, $n = 1$. Computing the derivative at $t=0$, we have~:
\begin{equation*}
\frac{1}{2}\left(\partial_t \E[S_1(f_t)^2]\right)_{|t=0} = \E[ (S_1(\L f)) (\omega) S_1(f)(\omega)],
\end{equation*}
and we obtain~:
\begin{multline*}
\frac{1}{2}\left(\partial_t \E[S_1(f_t)^2]\right)_{|t=0} = \E\big[ \{ \omega_{1,2}(f(\theta_2 \ \omega) - f(\theta_1 \ \omega)) + \omega_{-2,-1} (f(\theta_{-2} \ \omega) - f(\theta_{-1} \ \omega)) \} \\
(f(\theta_{-1} \ \omega) + f(\omega) + f(\theta_1 \ \omega)) \big].
\end{multline*}
We choose $f(\omega) = \omega_{-1,0} + (\omega_{2,3})^2$. Writing $\mu_i$ for $\E[(\omega_{0,1})^i]$, a computation shows that the latter is equal to~:
\begin{multline*}
\mu_1 \mu_2 - \mu_3 + \mu_4 - (\mu_2)^2 + 2 \mu_1 (\mu_2)^2 - 2 \mu_1 \mu_4
\ge \mu_4 (1-2 \mu_1) - \mu_3 - (\mu_2)^2 .
\end{multline*}
Letting $\eps > 0$ and $p \in [0,1]$, we choose the following law for $\omega_{0,1}$~:
$$
(1-p) \delta_0 + p (4+\eps) x^{-(5+\eps)} \1_{\{x \ge 1\}} \ \d x .
$$
Then, for $i \le 4$~:
$$
\mu_i = p \frac{4+\eps}{4+\eps-i} .
$$ 
If we choose $p = 1/4$, then $\mu_1 \le 1/3$ and it comes that~:
$$
\mu_4 (1-2 \mu_1) - \mu_3 - (\mu_2)^2 \ge \frac{1}{3} \mu_4 - \mu_3 - (\mu_2)^2 .
$$
Now taking $\eps$ close enough to $0$, $\mu_4$ becomes the dominant term, making the last expression strictly positive.

Yet, this example is inappropriate as it does not satisfy the uniform lower bound $\omega \ge 1$. It is clear that the law can be modified so that its support lies in $[\lambda , + \infty)$ for some $\lambda > 0$. Then changing $\omega$ for $\omega' = \omega/\lambda$, and taking $f'$ so that $f'(\omega') = f(\omega)$ completes this requirement. In a similar way, one can ensure that the support of $\omega$ is bounded, so that $\P$ satisfies ($\mathfrak{A}$), in which case $f'$ becomes bounded.
\end{proof}

\begin{proof}[Proof of Theorem~\ref{main:indep}]
Let $f : \Omega \to \R$ be such that $\mclN(f)$ is finite. Note first that for any $t \ge 0$, we have by Jensen's inequality $\E[(f^\circ_t)^2] \le \|f\|_2^2 \le \mclN(f)$. Then according to Proposition~\ref{contract}, for any $s \ge 0$~:
$$
\mathcal{N}(f^\circ_s) \le \mclN(f).
$$
Proposition \ref{decaygeneral} now gives that for any $t \ge 1$~:
$$
\E[(f^\circ_t)^2] \le C \mclN(f) (t-1)^{-d/2}.
$$
\end{proof}

%
%
%
%
\section{Resolvents comparison}
\label{s:resolvents}
\setcounter{equation}{0}

For any $\mu > 0$ and any $f \in L^2(\P)$, we define the resolvents as
\begin{equation}
\label{def:Rmu}
R_\mu f = (-\L + \mu)^{-1} f, \qquad R^\circ_\mu f = (-\L^\circ + \mu)^{-1} f.
\end{equation}
We write $(\cdot,\cdot)$ for the scalar product in $L^2(\P)$. We begin by recalling some classical results about resolvents.
\begin{lem}
\label{lemresolvent}
We have
$$
(R_\mu f, f) = \int_0^{+\infty} e^{-\mu t} \E[(f_{t/2})^2] \ \d t,
$$
and the same equality holds replacing $R_\mu, f_{t/2}$ by $R^\circ_\mu, f^\circ_{t/2}$. Moreover, the following comparison holds
$$
(R_\mu f, f) \le (R^\circ_\mu f, f).
$$
\end{lem}
\begin{proof}
The spectral theorem gives
\begin{eqnarray*}
(R_\mu f, f) & = & \int \frac{1}{\mu + \lambda} \ \d e_f(\lambda) \\
& = & \int \int_{0}^{+ \infty} e^{-(\mu + \lambda) t} \ \d t \ \d e_f(\lambda) \\
& = & \int_0^{+ \infty} e^{-\mu t} \int e^{- \lambda t} \ \d e_f(\lambda) \ \d t,
\end{eqnarray*}
using Fubini's theorem. The equality comes noting that
$$
\E[(f_{t/2})^2] = (f_{t/2},f_{t/2}) = (f_t,f) = \int e^{- \lambda t} \ \d e_f(\lambda).
$$
The second part of the lemma is due to \cite{blp} (see also \cite[Lemma 2.24]{woe}). We recall the proof here for convenience. Observe that $-\L^\circ + \mu$ defines a positive quadratic form. Hence, applying the Cauchy-Schwarz inequality, we get
\begin{eqnarray}
\label{calculresol}
(R_\mu f,f)^2 & = & (R_\mu f, (-\L^\circ + \mu) R^\circ_\mu f)^2 \notag \\ 
& \le & (R_\mu f, (-\L^\circ + \mu) R_\mu f) \ (R^\circ_\mu f, (-\L^\circ + \mu) R^\circ_\mu f).
\end{eqnarray}
Note that $(R_\mu f, -\L^\circ R_\mu f)$ is the Dirichlet form $\cE^\circ(R_\mu f,R_\mu f)$, which, as we saw before, is smaller than $\cE(R_\mu f,R_\mu f) = (R_\mu f,- \L  R_\mu f)$, thus leading to
$$
(R_\mu f, (-\L^\circ + \mu) R_\mu f) \le (R_\mu f, (-\L + \mu) R_\mu f) = (R_\mu f, f) .
$$
Using this result in equation~(\ref{calculresol}), one obtains
$$
(R_\mu f,f)^2 \le (R_\mu f, f) \ (R^\circ_\mu f, f),
$$
which proves the result.
\end{proof}

\begin{proof}[Proof of Theorem \ref{main:d<}]
As $\partial_t \E[(f_t)^2] = -2 \cE(f_t,f_t)$ is negative, the function $t \mapsto \E[(f_t)^2]$ is decreasing, so we have, for any $\mu > 0$~:
\begin{eqnarray*}
\E[(f_{t/2})^2] & \le & \frac{\int_0^t e^{-\mu s} \E[(f_{s/2})^2] \ \d s}{\int_0^t e^{-\mu s} \ \d s}  \\
& \le & \mu \frac{(R_\mu f,f)}{1-e^{-\mu t}}.
\end{eqnarray*}
Using Lemma \ref{lemresolvent}, we obtain~:
\begin{equation}
\label{compftR}
\E[(f_{t/2})^2] \le \mu \frac{(R^\circ_\mu f,f)}{1-e^{-\mu t}}.
\end{equation}

For $d=1$, Theorem \ref{main:indep} implies that there exists $C > 0$ such that for any $f$ with $\mclN(f)$ finite~:
\begin{eqnarray*}
(R^\circ_\mu f, f) & \le & C \mclN(f) \int_0^{+ \infty} \sqrt{\frac{2}{t}} e^{-\mu t} \ \d t \\
& \le & \frac{\sqrt{2} C \mclN(f)}{\sqrt{\mu}} \int_0^{+ \infty} \frac{e^{-u}}{\sqrt{u}} \ \d u.
\end{eqnarray*}
Using this bound in equation (\ref{compftR}), we get~:
$$
\E[(f_{t/2})^2] \le \frac{C \sqrt{\mu} \mclN(f)}{1-e^{-\mu t}}.
$$
Choosing $\mu = 1/t$ in the last expression gives the desired result.

For $d=2$, Theorem \ref{main:indep} implies that 
\begin{eqnarray*}
(R^\circ_\mu f, f) & \le & \|f\|_2^2 + C \mclN(f) \int_1^{+ \infty} \frac{2 e^{-\mu t}}{t} \ \d t \\
& \le & \|f\|_2^2 + C \mclN(f) \int_\mu^{+ \infty} {2 e^{-u}} \ \frac{\d u}{u}.
\end{eqnarray*}
Assuming that $\mu \le 1$ (and using the fact that $\mclN(f) \ge \|f\|_2^2$), it comes that
$$
(R^\circ_\mu f, f) \le \|f\|_2^2 + C \mclN(f) \left(C + \int_\mu^1 \frac{2 \d u}{u}\right) \le C \mclN(f) (1 + \ln(1/\mu)).
$$
By equation (\ref{compftR}), it comes that, for any $\mu \le 1$~:
$$
\E[(f_{t/2})^2] \le \mu C \mclN(f) \frac{1+ \ln(1/\mu)}{1-e^{- \mu t}}.
$$
The result comes choosing $\mu = 1/t$, whenever $t \ge 1$. As is always the case, $\E[(f_{t/2})^2]$ is controlled for smaller times by the bound $\E[(f_{t/2})^2] \le \|f\|_2^2 \le \mclN(f)$.

In larger dimension ($d \ge 3$), $(R^\circ_\mu f,f)$ remains bounded as $\mu$ goes to 0, and choosing $\mu = 1/t$ in equation (\ref{compftR}) gives the desired upper bound. Using the comparison lemma once again and Theorem~\ref{main:indep}, we also obtain~:
\begin{eqnarray*}
\int_0^{+ \infty} \E[(f_t)^2] \ \d t & \le & \int_0^{+ \infty} \E[(f^\circ_t)^2] \ \d t \\
& \le &  \int_0^{+ \infty} \frac{C \mclN(f)}{\max(1,t^{d/2})} \ \d t \\
& \le & C \mclN(f).
\end{eqnarray*}
\end{proof}

In the next section, we focus on finding upper bounds for $\E[S_n(f_t)^2]$. 
%
%
%
%
%
%

\section{A martingale method}
\setcounter{equation}{0}
\label{s:martmeth}

Recall the definition of $\tri \cdot \tri$ in (\ref{def:tri}). In this section, we will always assume, without further mention, that assumption $(\mathfrak{I})$ is satisfied. Under this assumption, we will prove the following upper bound on $\E[S_n(f_t)^2]$. 

\begin{thm}
\label{martmeth}
There exists $c \ge 0$ such that for any $f \in L^\infty(\P)$ with $\E[f] = 0$, any integer $n$ and any $t \ge 0$, we have~: 
\begin{equation*}
\E[S_n(f_t)^2] \le 2 \tri f \tri^2 |B_n| + c \|f\|_\infty^2 (1+t)^2 |B_n|.
\end{equation*}
\end{thm}

\begin{proof}
We choose an enumeration of the edges $\mathbb{B}^d = (e_k)_{k \in \N}$, and define $\mathcal{F}_k$ as the $\sigma$-algebra generated by $(\omega_{e_0},\ldots, \omega_{e_k})$, $\mathcal{F}_{-1} = \{\emptyset,\Omega\}$. For any $t \ge 0$, we define the martingale $M_k(t) = \E[S_n(f_t) | \mathcal{F}_k]$ and the corresponding martingale increments $\Delta_k(t) = M_k(t)-M_{k-1}(t)$. We have the following decomposition :
$$
S_n(f_t) = \sum_{k=0}^{+\infty} \Delta_k(t).
$$
This convergence holds almost surely. As $f \in L^\infty(\P)$, the dominated convergence theorem ensures that it holds also in $L^2$ sense. Due to the orthogonality of the increments, we get :
\begin{equation*}
\E[S_n(f_t)^2] = \sum_{k=0}^{+\infty} \E[\Delta_k(t)^2].
\end{equation*}
We will now estimate the right-hand side of this equality. First, we introduce a representation (that we learned from \cite{boivin}) for $M_k(t)$ that we find convenient. For two environments $\omega,\sigma \in \Omega$, we define~:
$$
\cro_k(e_i) = 
\left|
\begin{array}{l}
\omega_{e_i} \text{ if } i \le k \\
\sigma_{e_i} \text{ otherwise},
\end{array}
\right.
$$
and we write $\E_{\sigma}$ to refer to integration with respect to $\d \P(\sigma)$. As we assume in this section that $(\omega_e)_{e \in \mathbb{B}^d}$ are independent random variables, we can rewrite $M_k(t)$ as~:
\begin{eqnarray*}
M_k(t)(\omega) & = & \E_\sigma[S_n(f_t)(\cro_k)] \\
& = & \sum_{x \in B_n} \E_\sigma[f_t(\theta_x \ \cro_k)].
\end{eqnarray*}
But note that
$$
f_t(\theta_x \ \omega) = \EE_0^{\theta_x \omega}[f(\theta_{X_t+x} \ \omega)],
$$
and as the law of $X_t+x$ under $\PP_0^{\theta_x \omega}$ is the same as the one of $X_t$ under $\PPo_x$, we are led to
$$
f_t(\theta_x \ \omega) = \EEo_x[f(\theta_{X_t} \ \omega)],
$$
which implies that
$$
M_k(t) =  \sum_{x \in B_n} \E_\sigma\big[\EE^{\cro_k}_x[f(\theta_{X_t} \ \cro_k)]\big].
$$
We obtain~:
$$
\Delta_k(t) = \sum_{x \in B_n} \E_\sigma\big[\EE^{\cro_{k}}_x[f(\theta_{X_t} \ \cro_{k})] - \EE^{\cro_{k-1}}_x[f(\theta_{X_t} \ \cro_{k-1})]\big].
$$
Let $\mathcal{V}(e_k)$ be the set of endpoints of $e_k$. We want to distinguish whether the walk already met the set $\mathcal{V}(e_k)$ at time $t$ or not. For this purpose, we introduce the following events :
\begin{equation}
\label{defAt}
\mathcal{A}_t(y) = \left\{ \{X_s, 0\le s \le t\} \cap \mathcal{V}(e_k) = \emptyset  \text{ and } X_t = y \right\},
\end{equation}
\begin{equation}
\label{defA't}
\mathcal{A}'_t = \left\{ \{X_s, 0\le s \le t\} \cap \mathcal{V}(e_k) \neq \emptyset \right\}.
\end{equation}
Then one can decompose $\Delta_k(t)$ as $A_k(t)+A'_k(t)$, where~:
\begin{equation}
\label{akt}
\begin{split}
& A_k(t) \\
& \quad =  \sum_{\substack{x \in B_n \\ y \in \Z^d}} \E_\sigma\big[\EE^{\cro_{k}}_x[f(\theta_{X_t} \ \cro_{k}) \1_{\mathcal{A}_t(y)}] - \EE^{\cro_{k-1}}_x[f(\theta_{X_t} \ \cro_{k-1}) \1_{\mathcal{A}_t(y)}]\big] \\
& \quad = \sum_{\substack{x \in B_n \\ y \in \Z^d}} \E_\sigma \big[ f(\theta_y \ \cro_{k}) \PP^{\cro_{k}}_x[\mathcal{A}_t(y)] - f(\theta_y \ \cro_{k-1}) \PP^{\cro_{k-1}}_x[\mathcal{A}_t(y)] \big],
\end{split}
\end{equation}
and
\begin{equation}
\label{ak't}
A'_k(t) = \sum_{x \in B_n} \E_\sigma\big[\EE^{\cro_{k}}_x[f(\theta_{X_t} \ \cro_{k}) \1_{\mathcal{A}'_t}] - \EE^{\cro_{k-1}}_x[f(\theta_{X_t} \ \cro_{k-1}) \1_{\mathcal{A}'_t}]\big],
\end{equation}
so that we have~:
\begin{equation}
\label{decompSn(f)}
\E[S_n(f_t)^2] \le 2 \sum_{k=0}^{+\infty} \E[A_k(t)^2] + 2 \sum_{k=0}^{+\infty} \E[A_k'(t)^2].
\end{equation}
We now turn to evaluate each of these terms. Theorem~\ref{martmeth} is proved once we have shown the following results.
\begin{prop}
\label{prop:martmeth}
For any integer $n$ and any $t \ge 0$~:
\begin{equation*}
\sum_{k=1}^{+\infty} \E[A_k(t)^2] \le \tri f \tri^2 |B_n|.
\end{equation*}
There exists $c \ge 0$ (independent of $f$) such that for any integer $n$ and any $t \ge 0$~:
\begin{equation*}
\sum_{k=1}^{+\infty} \E[A'_k(t)^2] \le c \|f\|_\infty^2 (1+t)^2 |B_n|.
\end{equation*}
\end{prop}
\begin{proof}[Proof of Proposition \ref{prop:martmeth}]
Note that $\PPo_x[\mathcal{A}_t(y)]$, as a function of $\omega$, does not depend on $\omega_{e_k}$. Therefore, for almost every $\omega$, we have~:
\begin{eqnarray*}
A_k(t) & = & \sum_{\substack{x \in B_n \\ y \in \Z^d}} \E_\sigma \big[  (f(\theta_y \ \cro_{k})  - f(\theta_y \ \cro_{k-1})) \PP^{\cro_{k}}_x[\mathcal{A}_t(y)] \big] \\
|A_k(t)| & \le & \sum_{\substack{x \in B_n \\ y \in \Z^d}} | \grad f | (e_k-y)  \E_\sigma \big[ \PP^{\cro_{k}}_x[X_t = y] \big] ,
\end{eqnarray*}
where $e_k-y$ stands for the edge obtained when translating the edge $e_k$ by the vector $(-y)$ ($| \grad f |$ is defined in (\ref{def:grad})). Due to reversibility, we have that $\PP^{\cro_{k}}_x[X_t = y] = \PP^{\cro_{k}}_y[X_t = x]$, and it comes that~:
\begin{equation}
\label{barreAk}
|A_k(t)| \le \sum_{y \in \Z^d} | \grad f | (e_k-y) \E_\sigma \big[ \PP^{\cro_{k}}_y[X_t \in B_n] \big] .
\end{equation}
Using the fact that $\PP^{\cro_{k}}_y[X_t \in B_n] \le 1$, we get that for almost every $\omega$, $|A_k(t)| \le \tri f \tri$. Using this together with equation~(\ref{barreAk}), we obtain~:
$$
\sum_{k=1}^{+\infty} \E[A_k(t)^2]  \le  \sum_{k=1}^{+\infty} \tri f \tri \sum_{y \in \Z^d} | \grad f | (e_k-y) \E \E_{\sigma} \big[ \PP^{\cro_{k}}_y[X_t \in B_n] \big].
$$
Noting that the law of $\cro_{k}$ under $\E \E_{\sigma}$ is the same as the one of $\omega$ under $\E$, the latter simplifies into
$$
\tri f \tri \sum_{y \in \Z^d} \sum_{k=1}^{+\infty} | \grad f | (e_k-y) \E \big[ \PPo_y[X_t \in B_n] \big] 
\le \tri f \tri^2 \sum_{y \in \Z^d} \E \big[ \PPo_y[X_t \in B_n] \big].
$$
Using once again reversibility, we have 
$$
\sum_{y \in \Z^d} \E \big[ \PPo_y[X_t \in B_n] \big] = \sum_{\substack{x \in B_n \\ y \in \Z^d}} \E \big[ \PPo_x[X_t = y] \big] = |B_n| ,
$$ 
which shows the first part of the proposition.

We now turn to $A'_k(t)$. From equation (\ref{ak't}), we have the following estimate~:
\begin{eqnarray*}
|A'_k(t)| & \le & \|f\|_\infty \sum_{x \in B_n} \E_\sigma\big[\PP^{\cro_{k}}_x[\mathcal{A}'_t] + \PP^{\cro_{k-1}}_x[\mathcal{A}'_t]\big]  \notag \\
& \le & 2 \|f\|_\infty \sum_{x \in B_n} \E_\sigma\big[\PP^{\cro_{k}}_x[\mathcal{A}'_t] \big],
\end{eqnarray*}
where, in the last step, we used the fact that the event $\mathcal{A}'_t$ does not depend on the conductance $\omega_{e_k}$. If we define $\ov{A}_k(t)$ to be
\begin{equation}
\label{Akl}
\ov{A}_k(t) =  \sum_{x \in B_n} \PP^{\cro_{k}}_x[\mathcal{A}'_t],
\end{equation} 
then the last equation can be rewritten as
\begin{equation}
\label{A'kcomp}
|A'_k(t)| \le 2 \|f\|_\infty \E_\sigma[\ov{A}_k(t)].
\end{equation}

In order to estimate $\ov{A}_k(t)$, we would like to compare the probability to hit $\mathcal{V}(e_k)$ (the event $\mathcal{A}'_t$) with the expected time spent inside this set. However, because conductances may take arbitrarily large values, this comparison is not possible, as the walk may exit the set $\mathcal{V}(e_k)$ very fast. In order to circumvent this difficulty, we will compare the probability to hit $\mathcal{V}(e_k)$ with the expected time spent in a larger set, to be defined below. 

Let $p_\omega(x)$ be the total jump rate of site $x$~:
$$
p_\omega(x) = \sum_{y \sim x} \omega_{x,y}.
$$
For a reason that will become clear in the proof of part~(\ref{lemperco:percolepas}) of Lemma~\ref{lemperco}, we now introduce a parameter $\eta$, chosen large enough so that~:
\begin{equation}
\label{q:percolepas}
q := \P[p_\omega(0) > \eta] < \frac{1}{(2d)^{(2d+1)}}.
\end{equation}
We say that a point $x \in \Z^d$ is \emph{good} in the environment $\omega$ if $p_\omega(x) \le \eta$ ; we say that it is \emph{bad} otherwise. 


We now define $\ov{\mathcal{V}}_\omega(e_k)$ the following way~:
\begin{multline}
\label{gammal}
y \in \ov{\mathcal{V}}_\omega(e_k) \\
\Leftrightarrow y \in \mathcal{V}(e_k) \text{ or } \exists \gamma = (\gamma_1, \ldots, \gamma_l) : \gamma_1 \in \mathcal{V}(e_k), \gamma_l = y, \gamma_1,\ldots,\gamma_{l-1} \text{ bad points} ,
\end{multline}
where $\gamma$ is a nearest-neighbour path. The set $\ov{\mathcal{V}}_\omega(e_k)$ contains $\mathcal{V}(e_k)$, and any point in its inner boundary is a good point (in fact, it is the smallest such set). Indeed, the fact that it contains $\mathcal{V}(e_k)$ is clear from the definition. Moreover, let $y$ be a bad point in $\ov{\mathcal{V}}_\omega(e_k)$. Then there exists a nearest-neighbour path $\gamma_1,\ldots,\gamma_l$ as in (\ref{gammal}) (if $y \in \mathcal{V}(e_k)$, then one can choose the path that is made only of the point $y$). Observe that as $\gamma_l = y$, $\gamma_l$ is a bad point. Hence, if $z$ is a neighbour of $y$, considering the path $\gamma_1,\ldots,\gamma_l,z$, one can see that $z$ belongs to $\ov{\mathcal{V}}_\omega(e_k)$. We have shown that any bad point in $\ov{\mathcal{V}}_\omega(e_k)$ has all its neighbours in $\ov{\mathcal{V}}_\omega(e_k)$. This implies that any point in the inner boundary of $\ov{\mathcal{V}}_\omega(e_k)$ is a good point.

The following lemma relates the probability to hit ${\mathcal{V}}(e_k)$ with the expected time spent inside $\ov{\mathcal{V}}_\omega(e_k)$.

\begin{lem}
\label{lem:martmeth}
For every $x \in \Z^d$ and every $\omega$, we have :
\begin{equation}
\label{eq:lem:mm}
\PPo_x[\mathcal{A}'_t] \le e \eta \int_0^{t+1} \PPo_x[X_s \in \ov{\mathcal{V}}_\omega(e_k)] \ \d s .
\end{equation}
\end{lem}
\begin{proof}
Let $T$ be the hitting time of the set $\mathcal{V}(e_k)$~:
$$
T = \inf \{ s \ge 0 : X_s \in \mathcal{V}(e_k) \}.
$$
One can bound from below the integral appearing in the right hand side of (\ref{eq:lem:mm}) as follows~:
\begin{eqnarray*}
\int_0^{t+1} \PPo_x[X_s \in \ov{\mathcal{V}}_\omega(e_k)] \ \d s & \ge & \EEo_x\left[\1_{T\le t} \int_0^{t+1} \1_{ X_s \in \ov{\mathcal{V}}_\omega(e_k)} \ \d s \right] \\
& \ge & \EEo_x\left[\1_{T\le t} \int_T^{T+1} \1_{ X_s \in \ov{\mathcal{V}}_\omega(e_k)} \ \d s \right].
\end{eqnarray*}
By the Markov property at time $T$, the latter equals~:
$$
\EEo_x\left[\1_{T\le t} \ \EEo_{X_T} \left[ \int_0^{1} \1_{ X_s \in \ov{\mathcal{V}}_\omega(e_k)} \ \d s \right] \right].
$$
As $\PPo_x[T \le t] = \PPo_x[\mathcal{A}'_t]$, the lemma would be proved if we can show that
\begin{equation}
\label{lem:interm}
\EEo_{X_T} \left[ \int_0^{1} \1_{ X_s \in \ov{\mathcal{V}}_\omega(e_k)} \ \d s \right] \ge \frac{1}{e \eta}.
\end{equation}
Let $\ov{T}$ be the exit time from $\ov{\mathcal{V}}_\omega(e_k)$~:
$$
\ov{T} = \inf \{s \ge 0 : X_s \notin \ov{\mathcal{V}}_\omega(e_k) \}.
$$
The left hand side of (\ref{lem:interm}) is greater than
$$
\EEo_{X_T}[\min(\ov{T},1)].
$$
As we want to show that the exit from $\ov{\mathcal{V}}_\omega(e_k)$ occurs slowly enough, we are interested in the first time $T_g$ that the walk visits a good site~:
$$
T_g = \inf \{ s \ge 0 : X_s \text{ is a good site} \},
$$
and also in the time $\eps_g$ spent between $T_g$ and the next jump of the walk. Conditionally on $X_{T_g}$, the random variable $\eps_g$ is exponentially distributed, with parameter $p_\omega(X_{T_g})$. But by definition of $T_g$, the site $X_{T_g}$ is good, hence $p_\omega(X_{T_g}) \le \eta$, and we obtain that for any $y \in \Z^d$~:
\begin{equation}
\label{lem:exprv}
\PPo_y[\eps_g \ge \eta^{-1}] \ge e^{-1}.
\end{equation}
Moreover, if the walk starts from a site inside $\ov{\mathcal{V}}_\omega(e_k)$, then it is clear that $\ov{T} \ge \eps_g$. This is due to the fact that any site in the inner boundary of $\ov{\mathcal{V}}_\omega(e_k)$ is a good site, so the walk must meet a good site before exiting $\ov{\mathcal{V}}_\omega(e_k)$. We finally obtain~:
$$
\EEo_{X_T}[\min(\ov{T},1)] \ge \EEo_{X_T}[\min(\eps_g,1)] \ge \eta^{-1} \PPo_{X_T}[\eps_g \ge \eta^{-1}],
$$
which, together with (\ref{lem:exprv}), proves the lemma.
\end{proof}
Recalling the definition of $\ov{A}_k(t)$ from (\ref{Akl}), and using the estimate provided by the lemma, we obtain~:
\begin{equation}
\label{Aklineq}
\ov{A}_k(t) \le e \eta \sum_{x \in B_n} \int_0^{t+1} \PP^{\cro_{k}}_x[X_s \in \ov{\mathcal{V}}_{\cro_k}(e_k)] \ \d s .
\end{equation}
Using reversibility, one can rewrite it as~:
\begin{eqnarray}
\label{Aklineq2}
\ov{A}_k(t) & \le &  e \eta \int_0^{t+1} \sum_{y \in \ov{\mathcal{V}}_{\cro_k}(e_k)} \PP^{\cro_{k}}_y[X_s \in B_n] \ \d s \notag\\
& \le &  e \eta (t+1) |\ov{\mathcal{V}}_{\cro_k}(e_k)| .
\end{eqnarray}
Besides, by (\ref{A'kcomp}) and Jensen's inequality, note that~:
\begin{eqnarray}
\label{EA'kdecomp}
\E[A_k'(t)^2] & \le & 4 \|f\|_\infty^2 \E\left[ \left( \E_\sigma[\ov{A}_k(t)] \right)^2 \right] \notag \\
& \le & 4 \|f\|_\infty^2 \E \E_\sigma[\ov{A}_k(t)^2] .
\end{eqnarray}
Hence, in order to prove the second part of Proposition~\ref{prop:martmeth}, it is enough to show that there exists $c \ge 0$ such that for any integer $n$ and any $t \ge 0$~:
\begin{equation*}
\sum_{k=0}^{+\infty} \E \E_\sigma[\ov{A}_k(t)^2] \le c (t+1)^2 |B_n| .
\end{equation*}
Using inequalities (\ref{Aklineq}) and (\ref{Aklineq2}), we obtain~:
\begin{equation}
\label{EEsigmaovA}
\sum_{k=0}^{+\infty} \E \E_\sigma[\ov{A}_k(t)^2] 
\le  e^2 \eta^2 (t+1) \sum_{k=0}^{+\infty} \sum_{x \in B_n} \int_0^{t+1} \E  \left[|\ov{\mathcal{V}}_\omega(e_k)|  \PPo_x[X_s \in \ov{\mathcal{V}}_\omega(e_k)]  \right] \ \d s ,
\end{equation}
where the integration of $\cro_k$ under $\E \E_\sigma$ has been replaced by integration of $\omega$ under $\E$. 
In this last expression, we can rewrite the expectation the following way~:
\begin{equation}
\label{bientotW}
\sum_{k=0}^{+\infty} \E \left[|\ov{\mathcal{V}}_\omega(e_k)|  \PPo_x[X_s \in \ov{\mathcal{V}}_\omega(e_k)]  \right]
 = \sum_{k=0}^{+\infty} \sum_{y \in \Z^d} \E \left[|\ov{\mathcal{V}}_\omega(e_k)| \1_{y \in \ov{\mathcal{V}}_\omega(e_k)} \PPo_x[X_s = y] \right] .
\end{equation}
We introduce the following function of the environment~:
$$
W(\omega) = \sum_{k=0}^{+\infty} |\ov{\mathcal{V}}_\omega(e_k)| \1_{0 \in \ov{\mathcal{V}}_\omega(e_k)}.
$$
From the definition of $\ov{\mathcal{V}}_\omega(e_k)$ given in (\ref{gammal}), it is not hard to check that
$$
\ov{\mathcal{V}}_{\theta_y \omega}(e_k) = -y + \ov{\mathcal{V}}_\omega(y + e_k),
$$
where we understand $y + e_k$ as the edge obtained from $e_k$ by a translation of the vector $y$. This observation implies that~:
\begin{eqnarray*}
W(\theta_y \ \omega) & = & \sum_{k=0}^{+\infty} |\ov{\mathcal{V}}_\omega(y +  e_k)| \1_{y \in \ov{\mathcal{V}}_\omega(y + e_k)} \\
& = & \sum_{k=0}^{+\infty} |\ov{\mathcal{V}}_\omega(e_k)| \1_{y \in \ov{\mathcal{V}}_\omega(e_k)}.
\end{eqnarray*}
As a consequence, the right-hand side of equation~(\ref{bientotW}) becomes~:
\begin{equation}
\label{aborner}
\sum_{y \in \Z^d} \E \left[W(\theta_y \ \omega) \PPo_x[X_s = y] \right] = \E\EEo_x[W(\omega(s))].
\end{equation}
As the measure $\P$ is stationary for the environment viewed by the particle, this last expectation does not depend on $s$. We thus obtain that the expression appearing in the right hand side of (\ref{EEsigmaovA}) is equal to 
$$
\E[W(\omega	)] e^2 \eta^2 (t+1)^2 |B_n|,
$$
and Proposition~\ref{prop:martmeth} is proved, provided $\E[W(\omega)]$ is finite. We prove this fact in the next lemma.

Before stating it, we introduce $\mathcal{C}$ the set of $z \in \Z^d$ such that there exists a path from $0$ to $z$ visiting only bad points, except possibly $0$ and $z$~:
\begin{equation}
\label{defCy}
z \in \mathcal{C} \Leftrightarrow \exists \gamma = (\gamma_0, \ldots, \gamma_l) : \gamma_0 = 0, \gamma_l = z, \gamma_1,\ldots,\gamma_{l-1} \text{ bad points} ,
\end{equation}
where $\gamma$ is a nearest-neighbour path.

\begin{lem}
\label{lemperco}
\begin{enumerate}
\item 
\label{lemperco:Cy}
For any $\omega$, we have
\begin{equation}
\label{WleC2}
W(\omega) \le 2d |\mathcal{C}|^2.
\end{equation}
\item 
\label{lemperco:percolepas}
The random variable $W(\omega)$ is integrable.
\end{enumerate}
\end{lem}
\begin{proof}[Proof of Lemma~\ref{lemperco}]
To prove the first part of the lemma, it is enough to show the following implication~:
\begin{equation}
\label{implincl}
0 \in \ov{\mathcal{V}}_\omega(e_k) \Rightarrow \ov{\mathcal{V}}_\omega(e_k) \subseteq \mathcal{C}.
\end{equation}
Indeed, it would imply that
$$
W(\omega) \le |\{ k : \ov{\mathcal{V}}_\omega(e_k) \subseteq \mathcal{C} \}| \ |\mathcal{C}| \le |\{ k : {\mathcal{V}}(e_k) \subseteq \mathcal{C} \}| \ |\mathcal{C}|.
$$
Moreover, $|\{ k : {\mathcal{V}}(e_k) \subseteq \mathcal{C} \}|$ is the number of edges between any two points of $\mathcal{C}$. Any point having at most $2d$ edges attached to it, it is clear that this number is smaller than $2d|\mathcal{C}|$, and we obtain (\ref{WleC2}). We now proceed to prove (\ref{implincl}).

Let us assume that $0 \in \ov{\mathcal{V}}_\omega(e_k)$, and let $z \in \ov{\mathcal{V}}_\omega(e_k)$. We will show that $z \in \mathcal{C}$, distinguishing various cases and using the characterizations given in (\ref{gammal}) and (\ref{defCy}).

Suppose first that $0 \in \mathcal{V}(e_k)$. If $z$ belongs to $\mathcal{V}(e_k)$ as well, then considering the path $(0,z)$ (or simply $(0)$ if $z = 0$), one can see that $z \in \mathcal{C}$. Otherwise, there exists $\gamma' = (\gamma'_1, \ldots, \gamma'_{l'})$ such that $\gamma'_1 \in \mathcal{V}(e_k)$, $\gamma'_{l'} = z$, and $\gamma'_1,\ldots,\gamma'_{l'-1}$ are bad points. The path $(0, \gamma'_1, \ldots, \gamma'_{l'})$ (or $(\gamma'_1, \ldots, \gamma'_{l'})$ if $\gamma'_1 = 0$) satisfies the conditions needed to show that $z \in \mathcal{C}$. 

Now if $0 \notin \mathcal{V}(e_k)$, then there exists a path $\gamma = (\gamma_1, \ldots, \gamma_l)$ as in (\ref{gammal}). If $z$ is in $\mathcal{V}(e_k)$, then the path $(\gamma_l, \ldots, \gamma_1,z)$ (or $(\gamma_l, \ldots, \gamma_1)$ if $\gamma_1 = z$) shows that $z \in \mathcal{C}$. Otherwise, there exists $\gamma' = (\gamma'_1, \ldots, \gamma'_{l'})$ as before, and the path $(\gamma_l, \ldots, \gamma_1, \gamma'_1, \ldots, \gamma'_{l'})$ (or $(\gamma_l, \ldots, \gamma_1, \gamma'_2, \ldots, \gamma'_{l'})$ if $\gamma_1 = \gamma'_1$) is an appropriate choice. This finishes the proof of (\ref{implincl}).

We now turn to part (\ref{lemperco:percolepas}). For $x = (x_1,\ldots,x_d) \in \Z^d$, we define the graph norm $\|x\|_1 = |x_1|+\cdots+|x_d|$. Let $\mathcal{B}(r)$ be the ball centered at the origin and of radius $r$ with respect to $\|\cdot\|_1$. For any $r \in \N$~:
\begin{multline*}
\P[\mathcal{C} \nsubseteq \mathcal{B}(r)] \\ 
\le \P[ \exists \gamma = (\gamma_0,\ldots, \gamma_r) \text{ simple path w. } \gamma_0 = 0 \text{ and } \gamma_1, \ldots, \gamma_r \text{ bad points}].
\end{multline*}
Let $G_x$ be the event `$x$ is a good point'. Note that $G_x$ is independent of $(G_y)$, for all $y$ that is neither $x$ nor one of its neighbours (there are $2d+1$ such sites). So from $\{\gamma_i, 1\le i\le r\}$, we can extract a subset $\gamma'$ of cardinal at least $k = \lceil r/(2d+1) \rceil$, and such that $(G_{x})_{x \in \gamma'}$ are independent random variables. Recalling that we write $q$ for the probability to be a bad site, and bounding the number of possible paths by $(2d)^{r}$, we get~: 
$$
\P[\mathcal{C} \nsubseteq \mathcal{B}(r)] \le (2d)^{r} q^k.
$$
Using the hypothesis on $q$ made in (\ref{q:percolepas}), the quantity $\mu = 2d  q^{1/(2d+1)}$ is strictly smaller than $1$, and we obtain~:
$$
\P[\mathcal{C} \nsubseteq \mathcal{B}(r)] \le \mu^r.
$$
Also, note that, for some $C > 0$~:
$$
\P[|\mathcal{C}| \ge r] \le \P[\mathcal{C} \nsubseteq \mathcal{B}(C r^{1/d})].
$$
Combining these two facts, we obtain that
$$
\P[|\mathcal{C}| \ge r] \le \mu^{C r^{1/d}}.
$$
This decay of the tail probability ensures that $|\mathcal{C}|^2$ is integrable, and hence $W(\omega)$ as well as a consequence of the first part of the lemma.
\end{proof}
Part~(\ref{lemperco:percolepas}) of Lemma~\ref{lemperco} was the last required step in the establishment of Proposition~\ref{prop:martmeth}.
\end{proof}
Theorem~\ref{martmeth} now follows from Proposition~\ref{prop:martmeth}, as can be seen from inequality~(\ref{decompSn(f)}).
\end{proof}
%
%
%
%
%
%
%
%
\section{Theorem \ref{main:d>} and extensions}
\label{s:concl}
\setcounter{equation}{0}

We will now combine the results of the preceding sections to obtain the decay of the variance of $f_t$ as $t$ goes to infinity. 

\begin{proof}[Proof of Theorem \ref{main:d>}]
We recall that $\rN(f) = \tri f \tri^2 + \|f\|_\infty^2$. We have shown in Theorem~\ref{martmeth} that under assumption $(\mathfrak{I})$, there exists $c > 0$ such that for any $f \in L^\infty(\P)$ with $\E[f] = 0$, any integer $n$ and any $t \ge 0$~: 
\begin{equation*}
\E[S_n(f_t)^2] \le c \rN(f) (1+t)^2 |B_n|,
\end{equation*}
or in other terms, that $\mclN(f_t) \le  c \rN(f) (1+t)^2$. This, together with Proposition~\ref{decaygeneral}, implies that there exists $C > 0$ such that for any $t \ge 1$~:
$$
\E[(f_t)^2] \le C \rN(f) \left( \int_1^t (1+s)^{-4/d} \ \d s  \right)^{-d/2}.
$$
This estimate (together, for smaller times, with the fact that $\E[(f_t)^2] \le \|f\|_2^2 \le \rN(f)$) leads to the announced result.
\end{proof}

Our method of proof has the following direct generalization, which holds true without the additional assumption $(\mathfrak{I})$.

\begin{prop}
\label{generalization}
Let $\delta, \gamma \ge 0$ be such that $\delta > 2 \gamma$. There exists $C > 0$ such that, for any $f : \Omega \to \R$ satisfying, for any $n \in \N$ and any $t \ge 0$~:
$$
\E[(S_n(f_t))^2] \le c(f) n^{2d-\delta} (t+1)^\gamma
$$
where $c(f) > 0$, we have
$$
\E[(f_t)^2] \le C \frac{c(f) }{t^{\delta/2 - \gamma}}.
$$
Moreover, the claim still holds if one replaces $f_t$ by $f^\circ_t$.
\end{prop}
\begin{proof}
As given by Proposition~\ref{seminash}, we know that the above hypothesis implies 
$$
\E[(f_t)^2] \le C_S n^2 \cE(f_t,f_t) + c(f) n^{-\delta} (t+1)^\gamma.
$$
From this inequality (compare with equation~(\ref{presqnash})), one can follow the proof of Propositions~\ref{nash} and \ref{decaygeneral}, replacing $\mclN'(f_t)$ by $c(f) (t+1)^\gamma$ and the exponent $d$ by $\delta$. The fact that $c(f)$ is larger than $\|f\|_2^2$ is clear from the hypothesis (taking $n=0$ and $t=0$), and one obtains
$$
\E[f_t^2] \le C c(f) \left( \int_1^t s^{-2\gamma/\delta} \d s \right)^{-\delta/2} \le C \frac{c(f)}{t^{\delta/2 - \gamma}},
$$
provided $2 \gamma / \delta < 1$ and $t$ is large enough (say $t \ge 2$). The result now follows, using the bound $\E[f_t^2] \le \|f\|_2^2 \le c(f)$ for smaller times. The same proof applies to $f^\circ_t$ as well.
\end{proof}

The interest of this generalization is twofold. On one hand, it may provide a variance decay for functionals for which $\mclN(f)$ is infinite (for instance if the environment is assumed only to be mixing). On the other, it may strengthen the original claim for functionals for which $\E[(S_n(f_t))^2]$ decays atypically fast. For example, let $(e_1,\ldots,e_d)$ be the canonical basis of $\R^d$. We define the local drift in the first coordinate as
\begin{equation}
\label{deflocaldrift}
\mathfrak{d}(\omega) = \omega_{0,e_1} - \omega_{0,-e_1}.
\end{equation}
Then a simple calculation shows that in the sum $S_n(\mathfrak{d})$, most of the terms cancel out, except for boundary terms, so that, provided the conductances have a second moment, we have~:
$$
\E[(S_n(\mathfrak{d}))^2] \le C n^{d-1}.
$$
Now recall that, as given by Proposition~\ref{contract}, the function $t\mapsto \E[(S_n(\mathfrak{d}^\circ_t))^2]$ is decreasing, so the above estimate still holds for $\E[(S_n(\mathfrak{d}^\circ_t))^2]$, and Proposition~\ref{generalization} implies that~:
$$
\E[(\mathfrak{d}^\circ_t)^2] \le \frac{C}{t^{(d+1)/2}}.
$$
This exponent can in fact be improved to $d/2+1$, as we will see in the end of the last section.

We now present a second variation around Theorem~\ref{main:d>}. One drawback of this theorem is that is is often problematic to evaluate $\rN(f)$, when it is not simply infinite. For a particular class of functionals, we show below that one can get improved results with only minor changes in the method of proof.

For some $s \ge 0$, we say that a function $g(X,\omega)$ \emph{depends only on the trajectory up to time} $s$ if one can write it as a function of the sites visited up to time $s$, together with their neighbouring conductances, or more precisely if one can write $g(X,\omega)$ as
$$
g((X_u)_{u \le s},(\omega_{X_u + e})_{u \le s}),
$$
where we understand that $e$ ranges in the set of edges adjacent to $0$. We say that such a function is \emph{translation invariant} if moreover, for any $x \in \Z^d$~:
$$
g((x + X_u)_{u \le s},(\omega_{X_u + e})_{u \le s}) = g((X_u)_{u \le s},(\omega_{X_u + e})_{u \le s}).
$$
We write $\var(f)$ for the variance of the function $f$ with respect to the measure $\P$.
\begin{prop}
\label{gtraj}
Under assumption $(\mathfrak{I})$, there exists $C > 0$ such that, for any bounded function $g$ that depends only on the trajectory up to time $s$ and is translation invariant, if $f(\omega) = \EEo_0[g]$, then for any $t > 0$~:
\begin{eqnarray*}
\var(f_t) \le C \ \|g \|_\infty^2 \ln_+\left(\frac{s+t}{s}\right)^{-2} & \text{if } d = 4 , \\
\var(f_t) \le C \ \frac{\|g \|_\infty^2 (s+t)^2}{t^{d/2}} & \text{if } d \ge 5 .
\end{eqnarray*}
\end{prop}
\begin{proof}
We begin by recalling the definition of $f_t$, namely
$$
f_t(\omega) = \EEo_0\left[\EE^{\theta_{X_t} \omega}_0 [g((X_u)_{u \le s},(\omega_{X_u + e})_{u \le s})]\right].
$$
Recall also that the law of $X + x$ under $\PP^{\theta_x \omega}_0$ is the same as the one of $X$ under $\PPo_x$, therefore we have, using the translation invariance of $g$~:
$$
\EE^{\theta_{X_t} \omega}_0 [g((X_u)_{u \le s},(\omega_{X_u + e})_{u \le s})] = \EEo_{X_t} [g((X_u)_{u \le s},(\omega_{X_u + e})_{u \le s})].
$$
The Markov property now leads to
\begin{equation}
\label{g(t)}
f_t(\omega) = \EEo_0\left[g(t)\right],
\end{equation}
where, to ease notation, we wrote $g(t)$ for 
$$
g((X_u)_{t\le u \le t+s},(\omega_{X_u + e})_{t \le u \le t+s}).
$$

We define, as in section~\ref{s:martmeth}, $M_k(t) = \E[S_n(f_t) | \mathcal{F}_k]$, and the martingale increments $\Delta_k(t) = M_k(t) - M_{k-1}(t)$. Writing $\ov{f}$ for $f - \E[f]$, we obtain~:
$$
S_n(\ov{f}_t) = S_n(f_t) - \E[S_n(f_t)] = \sum_{k = 0}^{+ \infty} \Delta_k(t).
$$
One has the following expression for $\Delta_k(t)$~:
$$
\Delta_k(t) =  \sum_{x \in B_n} \E_\sigma\left[ \EE^{\cro_k}_x[g(t)] - \EE^{\cro_{k-1}}_x[g(t)] \right] ,
$$
which was the starting point of the computations of section~\ref{s:martmeth}. However, we now introduce the events~:
$$
\td{\mathcal{A}}_t = \left\{ \{X_u, 0\le s \le t + s\} \cap \mathcal{V}(e_k) = \emptyset \right\},
$$
and $\td{\mathcal{A}}'_t$ the complementary event. Note the difference with (\ref{defAt}-\ref{defA't}) : we consider intersections with $\mathcal{V}(e_k)$ up to time $t + s$, and not just up to time $t$. Then one can decompose $\Delta_k(t)$ as $\td{A}_k(t)+\td{A}'_k(t)$, where~:
\begin{equation*}
\td{A}_k(t) =  \sum_{x \in B_n} \E_\sigma\big[\EE^{\cro_{k}}_x[g(t) \1_{\td{\mathcal{A}}_t}] - \EE^{\cro_{k-1}}_x[g(t) \1_{\td{\mathcal{A}}_t}]\big] 
\end{equation*}
and
\begin{equation*}
\td{A}'_k(t) = \sum_{x \in B_n} \E_\sigma\big[\EE^{\cro_{k}}_x[g(t) \1_{\td{\mathcal{A}}'_t}] - \EE^{\cro_{k-1}}_x[g(t) \1_{\td{\mathcal{A}}'_t}]\big].
\end{equation*}

The remarkable property is that one has~:
$$
\EE^{\cro_{k}}_x[g(t) \1_{\td{\mathcal{A}}_t}] = \EE^{\cro_{k-1}}_x[g(t) \1_{\td{\mathcal{A}}_t}].
$$
Indeed, the event $\td{\mathcal{A}}_t$ itself does not depend on the value of $\omega_{e_k}$, and on this event, the law of the walk up to time $t+s$ does not either, nor does $g(t)$ by the definition of $g$ as a function that depends only on the trajectory up to time $s$.
Therefore $\td{A}_k(t)$ simply vanishes in that case. The evaluation of $\td{A}'_k(t)$ follows the same line as for Proposition~\ref{prop:martmeth}, and one finally obtains that there exists $C > 0$ (independent of $g$, $s$ and $t$) such that~:
$$
\mclN(\ov{f}_t) \le C \|g\|_\infty^2 (s + t + 1)^2.
$$
From this estimate, and using Proposition~\ref{decaygeneral}, it comes that, for any $t > 1$~:
$$
\E\left[(\ov{f}_t)^2\right] \le C \|g\|_\infty^2  \left( \int_1^{t} (s + u + 1)^{-4/d} \ \d u \right)^{-d/2}.
$$
Elementary computations then lead to the claims of the proposition, using the fact that $\var(f_t) \le \|g\|_\infty^2$ to treat smaller times. 
\end{proof}
Proposition~\ref{gtraj} is used in \cite{scaling} in order to obtain the scaling limit of the random walk among random traps.
%
%
%
%
%
%
\section{Central limit theorems}
\label{s:clt}
\setcounter{equation}{0}
We begin with the proof of Theorem \ref{tcl}, which shows the equivalence between the algebraic decay of the variance of $f_t$ and a particular behaviour of the spectral measure $e_f$ around $0$.
\begin{proof}[Proof of Theorem \ref{tcl}]
Note that the variance of $f_t$ can be rephrased in terms of the spectral measure $e_f$ as  
$$
\E[(f_t)^2] = \int e^{-2 \lambda t} \d e_f(\lambda) .
$$
The variance decay given in point (\ref{tcl:equiv1}) of Theorem \ref{tcl} is therefore equivalent to~:
\begin{equation}
\label{laplace}
\int e^{-\lambda t} \d e_f(\lambda) \le \frac{C}{t^\alpha}.
\end{equation}
We will show that equation~(\ref{laplace}) is equivalent to point (\ref{tcl:equiv3}) of the theorem.

Assume that equation~(\ref{laplace}) holds, and let 
$$
\chi(t) = 
\left|
\begin{array}{ll}
0 & \text{if } 0 \le t < 1 \\
1-e^{-t} & \text{if } t \ge 1 .
\end{array}
\right.
$$
Multiplying inequality~(\ref{laplace}) by $\chi(\delta t)$ and integrating over $t \in [0,+\infty)$, it comes that
\begin{eqnarray}
\label{laplint1}
\int_{t=0}^{+\infty} \chi(\delta t) \int e^{-\lambda t} \d e_f(\lambda) \d t & \le & C \int_0^{+\infty} \chi(\delta t) t^{-\alpha} \d t \notag \\
& \le & C \delta^{\alpha-1} \int_1^{+\infty} \chi(u) u^{-\alpha} \d u,
\end{eqnarray}
where the integral is finite, as $\alpha > 1$ and $\chi$ is bounded. Now, $\chi$ has been chosen so that 
$$
\int _{t=0}^{+\infty} e^{-\lambda t} \chi(\delta t) \d t = \frac{e^{-\lambda/\delta}}{\lambda} - \frac{e^{-(\lambda+\delta)/\delta}}{\lambda + \delta} \ge \frac{1}{2 e \lambda} \1_{[0,\delta]}(\lambda),
$$
so we get, using Fubini's theorem~:
\begin{equation}
\label{laplint2}
\int_{t=0}^{+\infty} \int \chi(\delta t) e^{-\lambda t} \d e_f(\lambda) \d t \ge \int_{[0,\delta]} \frac{1}{2 e \lambda} \d e_f(\lambda).
\end{equation}
Combining (\ref{laplint1}) and (\ref{laplint2}), we obtain~:
\begin{equation}
\label{equiv22}
\int_{[0,\delta]} \frac{1}{\lambda} \d e_f(\lambda) \le C \delta^{\alpha-1}.
\end{equation}
Reciprocally, assume that (\ref{equiv22}) holds. Note that 
$$
\lambda e^{-\lambda t} = \int_\lambda^{+\infty} (\delta t-1)e^{-\delta t} \d \delta .
$$
It comes that
\begin{eqnarray*}
\int e^{-\lambda t} \d e_f(\lambda) & = & \int \int_{\delta =0}^{+\infty} (\delta t-1)e^{-\delta t} \frac{1}{\lambda} \1_{\{\lambda \le \delta \}} \ \d \delta  \ \d e_f(\lambda) \\
& = & \int_{0}^{+\infty} (\delta t-1)e^{-\delta t} \int_{[0,\delta]} \frac{1}{\lambda} \d e_f(\lambda) \  \d \delta ,
\end{eqnarray*}
the use of Fubini's theorem being justified by the fact that the integrand is bounded in absolute value by $(\delta t + 1) e^{-\delta t}/{\lambda}$, which is integrable (the total mass of $e_f$ is $\|f\|_2^2$, assumed to be finite). Using inequality (\ref{equiv22}), the former is bounded by
$$
C \int_{0}^{+\infty} (\delta t-1)e^{-\delta t} \delta^{\alpha-1} \d \delta = \frac{C}{t^\alpha} \int_{0}^{+\infty} (u-1)e^{-u} u^{\alpha-1} \d u,
$$
which proves that (\ref{laplace}) holds.
\end{proof}
We now present some results taken from \cite{kipvar}. Although the authors give a complete proof for the discrete time case, they provide less details in the continuous time setting. For convenience of the reader, and also because we will need some of these details in the sequel, we provide here part of the proof of these results. We recall that the definition of $Z_t$ was given in (\ref{defZt}).
%
%
\begin{thm}[\cite{kipvar}]
\label{thmkipvar}
If
\begin{equation}
\label{condkipvarrappel}
\int \frac{1}{\lambda} \ \d e_f(\lambda) < + \infty,
\end{equation}
then there exist $(M_t)_{t \ge 0}$, $(\xi_t)_{t \ge 0}$ such that $Z_t = M_t + \xi_t$, where $(M_t)$ is a martingale with stationary increments under $\ov{\P}$ (and the natural filtration), and $(\xi_t)$ is such that~:
\begin{equation}
\label{e:xitendvers0}
\frac{1}{t} \ov{\E}[(\xi_t)^2] = 2 \int \frac{1-e^{-\lambda t}}{\lambda^2 t} \ \d e_f(\lambda) \xrightarrow[t \to + \infty]{} 0.
\end{equation}
Moreover, if $\P$ is ergodic, $(\eps Z_{t/\eps^2})_{t \ge 0}$ converges, as $\eps$ goes to $0$, to a Brownian motion of variance
$$
\sigma^2 = \ov{\E}[(M_1)^2] = 2 \int \frac{1}{\lambda} \ \d e_f(\lambda).
$$
\end{thm}
\noindent \textbf{Remark.} We shall only give the proof of the first part of the theorem. We refer to \cite{kipvar} for a proof of the invariance principle. 
\begin{proof}
Let $\eps > 0$, and let $u_{\eps}$ be such that
\begin{equation}
\label{defueps}
(-\L + \eps) u_\eps = f.
\end{equation}
We define $(M_t^\eps)$, $(\xi_t^\eps)$, $(\eta_t^\eps)$ by :
\begin{eqnarray*}
M_t^\eps & = & u_\eps(\omega(t)) - u_\eps(\omega(0)) - \int_0^t \L u_\eps(\omega(s)) \d s \\
\xi_t^\eps & = & - u_\eps(\omega(t)) + u_\eps(\omega(0)) \\
\eta_t^\eps & = & \int_0^t \eps u_\eps(\omega(s)) \d s.
\end{eqnarray*}
Using the definition of $u_\eps$ in (\ref{defueps}), one can see that $Z_t = M_t^\eps + \xi_t^\eps + \eta_t^\eps$. Moreover, $(M_t^\eps)_{t \ge 0}$ is a martingale with stationary increments under $\ov{\P}$. We first show that $\eta^\eps_t$ tends to $0$ in $L^2(\ov{\P})$ as $\eps$ goes to $0$. We will then prove that $(M_t^\eps)$, $(\xi_t^\eps)$ have limits in $L^2(\ov{\P})$ as $\eps$ goes to zero, by checking that they are Cauchy sequences. Writing these two limits $M_t$ and $\xi_t$ respectively, we will have $Z_t = M_t + \xi_t$.

We begin by showing that $\eta^\eps_t$ tends to $0$ as $\eps$ goes to $0$. Note that
\begin{eqnarray*}
\ov{\E}\left[ \left( \int_0^t g (\omega(s)) \d s \right)^2 \right] & = & 2 \int_{0 \le s \le u \le t} \ov{\E}[ g(\omega(s)) g(\omega(u))] \d s \ \d u\\
& = & 2 \int_{0 \le s \le u \le t} \ov{\E}[ g(\omega(0)) g(\omega(u-s))] \d s \ \d u,
\end{eqnarray*}
using the stationarity of $(\omega(s))$. By a change of variables (and using the fact that $\ov{\E} = \E.\EEo_0$), the latter becomes
\begin{equation}
\label{petitot2}
2 \int_0^t (t-s) \E[ g(\omega) g_s(\omega)] \d s.
\end{equation}
Together with the fact that $u_\eps = (-\L + \eps)^{-1} f$, this enables us to compute
\begin{eqnarray*}
\ov{\E}[(\eta_t^{\eps})^2] & = & 2 \int \int_0^t (t-s) e^{-\lambda s} \d s \left( \frac{\eps}{\lambda+\eps} \right)^2 \d e_f(\lambda) \\
& = & 2 \int \frac{e^{-\lambda t} - 1 + \lambda t}{\lambda^2} \frac{\eps^2}{(\lambda + \eps)^2} \d e_f(\lambda).
\end{eqnarray*}
As $\lambda \mapsto (e^{-\lambda t} - 1 + \lambda t)/\lambda^2$ remains bounded (and does not depend on $\eps$), the dominated convergence theorem shows that
$$
\ov{\E}[(\eta_t^{\eps})^2] \xrightarrow[\eps \to 0]{} 0.
$$

We now turn to the study of $(M_t^\eps)$. We will show that
\begin{equation}
\label{computM}
\ov{\E}\left[ \left( g(\omega(t)) - g(\omega(0)) - \int_0^t \L g (\omega(s)) \d s \right)^2 \right] = -2 t \E[g(\omega) \L g (\omega)].
\end{equation}
Indeed, as noted before, the process
$$
g(\omega(t)) - g(\omega(0)) - \int_0^t \L g (\omega(s)) \d s
$$
is a martingale with stationary increments under $\ov{\P}$, so the expression on the left hand side of equation (\ref{computM}) is equal to $Ct$ for some $C \ge 0$. $C$ can be determined by computing the derivative at $t = 0$. Note that the the left hand side of (\ref{computM}) can be decomposed into
\begin{multline*}
\ov{\E}\left[ \left( g(\omega(t)) - g(\omega(0)) \right)^2\right] - 2 \ov{\E}\left[ \left( g(\omega(t)) - g(\omega(0)) \right) \int_0^t \L g (\omega(s)) \d s \right]  \\ + \ov{\E}\left[ \left(\int_0^t \L g (\omega(s)) \d s \right)^2 \right].
\end{multline*}
Due to the reversibility of the process, (the cadlag modification of) $(\omega(t-s))_{0 \le s \le t}$ has the same law under $\ov{\P}$ as $(\omega(s))_{0 \le s \le t}$. But under this time reversal, the integral remains unchanged while $g(\omega(t)) - g(\omega(0))$ is changed into its opposite. The double product is therefore equal to zero. As for the square of the integral, the computation that lead to equation~(\ref{petitot2}) shows that it is bounded by a constant times $t^2$. In particular, its derivative at $0$ is equal to zero, and there is only the first term left. Its derivative at $0$ is known to be twice the Dirichlet form $\cE(g,g) = -\E[g(\omega) \L g (\omega)]$.
It comes that
\begin{eqnarray*}
\ov{\E}[(M_t^{\eps_1} - M_t^{\eps_2})^2] & = & \int 2 t \lambda \left( \frac{1}{\lambda+\eps_1} - \frac{1}{\lambda + \eps_2} \right)^2 \d e_f(\lambda) \\
& = & \int \frac{2 t \lambda (\eps_2-\eps_1)^2}{(\lambda+\eps_1)^2(\lambda + \eps_2)^2} \d e_f(\lambda).
\end{eqnarray*}
If, say, $\eps_2 > \eps_1$, then the integrand is bounded by 
$$
\frac{2 t \lambda (\eps_2)^2}{\lambda^2 (\eps_2)^2} = \frac{2t}{\lambda},
$$
which is integrable due to assumption (\ref{condkipvarrappel}). So dominated convergence theorem applies, and as
$$
\frac{2 t \lambda (\eps_2-\eps_1)^2}{(\lambda+\eps_1)^2(\lambda + \eps_2)^2} \le \frac{2 t \lambda (\eps_2)^2}{\lambda^4} \xrightarrow[\eps_1,\eps_2 \to 0]{} 0,
$$
we get
$$
\ov{\E}[(M_t^{\eps_1} - M_t^{\eps_2})^2] \xrightarrow[\eps_1,\eps_2 \to 0]{} 0.
$$
We write $M_t$ for the $L^2(\ov{\P})$ limit of $(M_t^\eps)_{\eps > 0}$ as $\eps$ goes to $0$.

What is left is to check the convergence of $\xi^\eps_t$.
Similarly, it comes that
\begin{eqnarray*}
\ov{\E}[(\xi_t^{\eps_1} - \xi_t^{\eps_2})^2] & = & \int 2 (1-e^{-\lambda t}) \left( \frac{1}{\lambda+\eps_1} - \frac{1}{\lambda + \eps_2} \right)^2 \d e_f(\lambda) \\
& = & \int \frac{2 (1-e^{-\lambda t}) (\eps_2-\eps_1)^2}{(\lambda+\eps_1)^2(\lambda + \eps_2)^2} \d e_f(\lambda).
\end{eqnarray*}
and we show the same way, using dominated convergence theorem, that the latter converges to $0$ as $\eps_1,\eps_2$ go to $0$. We write $\xi_t$ for the limit. We have
$$
\ov{\E}[(\xi_t^{\eps})^2] = \int 2 (1-e^{-\lambda t}) \left( \frac{1}{\lambda+\eps} \right)^2 \d e_f(\lambda).
$$
Letting $\eps$ go to $0$, the $L^2$ convergence on one hand, and the monotone convergence theorem on the other, ensure that~:
\begin{equation*}
\frac{1}{t} \ov{\E}[(\xi_t)^2] = 2\int \frac{1-e^{-\lambda t}}{\lambda^2 t} \ \d e_f(\lambda).
\end{equation*}
Using the fact that for any $x \ge 0$, $1-e^{-x} \le x$, the integrand above is bounded by $1/\lambda$, which is assumed to be integrable. Dominated convergence theorem thus implies that the  integral above converges to $0$ as $t$ tends to infinity.
\end{proof}
Under a stronger assumption, the next proposition gives an estimate of the speed of convergence to $0$ in equation~(\ref{e:xitendvers0}).
\begin{prop}
\label{propdecay}
Under one of the equivalent conditions (\ref{tcl:equiv1}), (\ref{tcl:equiv3}) of Theorem~\ref{tcl}, there exists $C > 0$ such that 
\begin{equation}
\label{decayxi}
\frac{1}{t} \ov{\E}[(\xi_t)^2] = 2 \int \frac{1-e^{-\lambda t}}{\lambda^2 t} \ \d e_f(\lambda) \le \frac{C}{\psi_\alpha(t)},
\end{equation}
where $\psi_\alpha$ is defined in (\ref{defpsi}).
\end{prop}

\begin{proof}
The equality of the first two terms in (\ref{decayxi}) was given in Theorem~\ref{thmkipvar}. For the inequality, note first that for any $\alpha$, we have
$$
\psi_\alpha(t) \le t \qquad \text{and} \qquad \psi_\alpha(t) \le t^{\alpha-1}.
$$
Assuming that $t \ge 1$, we will decompose the interval of integration $\R_+$ into $[0,1/t) \cup [1/t,1) \cup [1,+\infty)$. As for any $x \ge 0$, $1-e^{-x} \le x$, we have
$$
\int_0^{1/t} \frac{1-e^{-\lambda t}}{\lambda^2 t} \ \d e_f(\lambda) \le C \int_0^{1/t} \frac{1}{\lambda} \ \d e_f(\lambda) \le \frac{C}{t^{\alpha-1}},
$$
using condition (\ref{tcl:equiv3}) of Theorem~\ref{tcl}. On the other hand,
$$
\int_1^{+\infty} \frac{1-e^{-\lambda t}}{\lambda^2 t} \ \d e_f(\lambda) \le \int_1^{+\infty} \frac{1}{t} \ \d e_f(\lambda) \le \frac{\|f\|_2^2}{t}.
$$
Now for the integral between $1/t$ and $1$, it is bounded from above by the following, on which we perform a kind of integration by parts~:
\begin{eqnarray*}
\int_{1/t}^1 \frac{1}{\lambda t} \frac{1}{\lambda} \ \d e_f(\lambda) & = & \int_{\lambda = 1/t}^1 \int_{\delta=\lambda}^{+\infty} \frac{1}{\delta^2 t} \frac{1}{\lambda} \ \d \delta \ \d e_f(\lambda) \\
& = & \int_{\delta=1/t}^{+\infty} \frac{1}{\delta^2 t} \int_{\lambda=1/t}^{\min(\delta,1)}  \frac{1}{\lambda} \ \d e_f(\lambda) \ \d \delta.
\end{eqnarray*}
Using property (\ref{tcl:equiv3}) of Theorem \ref{tcl} once again, it comes that the latter is bounded by
$$
C \int_{\delta=1/t}^{+\infty} \frac{1}{\delta^2 t} \min(\delta,1)^{\alpha-1} \d \delta.
$$
This integral can be decomposed into two parts~:
$$
\int_{\delta=1}^{+\infty} \frac{1}{\delta^2 t} \d \delta = \frac{1}{t},
$$
and
$$
\int_{\delta=1/t}^{1} \frac{\delta^{\alpha-3}}{t}  \d \delta = 
\left|
\begin{array}{ll}
\frac{1}{(\alpha-2)t} \left( 1 - \frac{1}{t^{\alpha - 2}} \right) & \text{if } \alpha \neq 2 \\
\frac{\ln(t)}{t} & \text{if } \alpha = 2. 
\end{array}
\right.
$$
which proves the proposition.
\end{proof}
\begin{prop}
\label{decayZ2}
Under one of the equivalent conditions (\ref{tcl:equiv1}), (\ref{tcl:equiv3}) of Theorem~\ref{tcl}, there exists $C > 0$ such that 
\begin{equation*}
0 \le \sigma^2 - \frac{1}{t} \ov{\E}\left[\left( Z_t \right)^2\right] \le \frac{C}{\psi_\alpha(t)}.
\end{equation*}
where $\sigma$ is defined by (\ref{defsigma}).
\end{prop}

\begin{proof}
Note that
$$
\ov{\E}[(M_t)^2] = \ov{\E}[(Z_t - \xi_t)^2] = \ov{\E}[(Z_t)^2] - 2 \ov{\E}[Z_t \xi_t] + \ov{\E}[(\xi_t)^2].
$$
Stationarity of the increments of $M_t$ implies that $\ov{\E}[(M_t)^2] = t \sigma^2$. Due to reversibility, (the cadlag modification of) $(\omega(t-s))_{0 \le s \le t}$ has the same law under $\ov{\P}$ as $(\omega(s))_{0 \le s \le t}$. But under this time reversal, $Z_t$ remains unchanged while $\xi_t$ is changed into $-\xi_t$ (it is enough to check that it is true on $\xi^\eps_t$, which is clear). Therefore, we have
$$
\ov{\E}[Z_t \xi_t] = - \ov{\E}[Z_t \xi_t] = 0.
$$
Proposition \ref{propdecay} states that
$$
0 \le \frac{1}{t} \ov{\E}[(\xi_t)^2] \le \frac{C}{\psi_\alpha(t)},
$$
which proves the proposition.
\end{proof}
We write $X_t \in \Z^d$ as $(X_{1,t},\ldots,X_{d,t})$. Corollary \ref{tclcor} is implied by the following result.
\begin{prop}
\label{prop:tclcor}
Let $i \in \{1,\ldots,d\}$. Under assumption $(\mathfrak{A})$ and if $d \ge 7$, there exists $C > 0$ such that
$$
0 \le \frac{1}{t} \ov{\E}\left[(X_{i,t})^2\right] - \ov{\sigma}^2 \le \frac{C}{\psi_{\alpha}(t)} \qquad \text{ with } \alpha = \frac{d}{2}-2.
$$
\end{prop}
\begin{proof}
By symmetry, it is sufficient to prove the result for $i = 1$. Recall from equation~(\ref{deflocaldrift}) that we defined the local drift in the first coordinate as
$$
\mathfrak{d}(\omega) = \omega_{0,e_1} - \omega_{0,-e_1}.
$$
Let 
$$
Z_t = \int_0^t \mathfrak{d}(\omega(s)) \d s.
$$
Then
$$
N_t := X_{1,t} - Z_t
$$
is a martingale with stationary increments under $\ov{\P}$. Under assumption ($\mathfrak{A}$), it is clear that $\rN(\mathfrak{d})$ is finite, as it is bounded and depends only on a finite number of coordinates. Using Theorem~\ref{main:d>}, we obtain that assumption (\ref{tcl:equiv1}) of Theorem \ref{tcl} is satisfied with $\alpha = d/2-2$. We have $\alpha > 1$ whenever $d \ge 7$, and in this case, Proposition~\ref{decayZ2} applies~: 
\begin{equation}
\label{proofcordecay}
0 \le {\sigma}^2 - \frac{1}{t} \ov{\E}\left[(Z_t)^2\right] \le \frac{C}{\psi_{\alpha}(t)},
\end{equation}
where $\sigma^2 = 2 \int \lambda^{-1} \ \d e_\mathfrak{d}(\lambda)$.
We obtain
\begin{equation}
\label{decompM'}
\ov{\E}[(N_t)^2] =  \ov{\E}[(X_{1,t})^2] - 2 \ov{\E}[X_{1,t} Z_t] + \ov{\E}[(Z_t)^2].
\end{equation}
Stationarity of the increments of $N_t$ implies that $\ov{\E}[(N_t)^2] = t (\sigma')^2$ for some $\sigma' \ge 0$. We will now show that $\ov{\E}[X_{1,t} Z_t] = 0$.

Indeed, one can see $X_{1,t}$ as a functional of $(\omega(s))_{0 \le s \le t}$ (this is valid whenever $\omega$ is not periodic, which is true almost surely). But as we saw before, the (cadlag modification of the) time reversal $(\omega(t-s))_{0 \le s \le t}$ has the same law as $(\omega(s))_{0 \le s \le t}$ under $\ov{\P}$. It is clear that this time reversal changes $X_{1,t}$ into $-X_{1,t}$. On the other hand, it leaves $Z_t$ unchanged. Therefore, we obtain
$$
\ov{\E}[X_{1,t} Z_t] = - \ov{\E}[X_{1,t} Z_t] = 0,
$$
which, together with (\ref{proofcordecay}) and (\ref{decompM'}), proves the proposition (with $\ov{\sigma}^2 = (\sigma')^2 - \sigma^2$).
\end{proof}

%
%
%
%
%
%
\section{Addendum}
\label{s:addendum}
\setcounter{equation}{0}
A referee made us aware of the recent works \cite{gloria1,gloria2}, in which the authors propose a practical way to estimate the effective diffusion matrix. As we will now see, their results have consequences in terms of the exponent of decay to equilibrium associated with the function $\mathfrak{d}$ defined in (\ref{deflocaldrift}). They consider the case when $d \ge 2$. For bounded, independent and identically distributed conductances, our Theorems~\ref{main:d<} and \ref{main:d>} give an exponent of decay equal to $\max(1,d/2 - 2)$ (with a logarithmic correction for $d = 2$). On the other hand, roughly speaking, their results imply that the exponent of decay is at least $\min(d/2 + 1, 3)$ for this particular function $\mathfrak{d}$, which is a better result when $d \le 9$. This observation enables to strengthen Corollary~\ref{tclcor} when $d \le 8$. Our approach does not provide such good exponents, but has the advantage of covering a large class of functions at once, and also gives some information when the conductances are unbounded (although in this case, Theorem~\ref{main:d>} does not apply to the function $\mathfrak{d}$, as $\rN(\mathfrak{d})$ becomes infinite).

We describe their approach briefly, and refer the reader to \cite{gloria1,gloria2} for details. The authors assume that the conductances are bounded, independent and identically distributed. In this case, we recall that the effective diffusion matrix is $\ov{\sigma}^2$ times the identity matrix. The authors ask for a practical way to compute $\ov{\sigma}^2$ numerically, with a control of the error. The constant $\ov{\sigma}^2$ can be expressed in terms of the ``corrector field''. We recall from (\ref{deflocaldrift}) the definition of the local drift in the first coordinate as
\begin{equation}
\label{deffung}
\mathfrak{d}(\omega) = \omega_{0,e_1} - \omega_{0,-e_1}.
\end{equation}
The corrector is a function $\phi : \Omega \to \R$ such that 
\begin{equation}
\label{Poissonequation}
-\L \phi = \mathfrak{d},
\end{equation}
whose gradient is stationnary and of mean $0$ \cite[Theorem 3]{kun}.

Let us define the following quantities~:
$$
A_0(\phi) = \E[\omega_{0,e_1}] - \cE(\phi,\phi),
$$
$$
A_1(\phi) = \E[\omega_{0,e_1}(1 + \phi(\theta_{e_1} \ \omega) - \phi(\omega))],
$$
\begin{equation}
\label{gloriaotto}
A_2(\phi) = \frac{1}{2} \sum_{|z| = 1} \E[\omega_{0,z}(e_1 \cdot z + \phi(\theta_{z} \ \omega) - \phi(\omega))^2].
\end{equation}
These three expressions are all equal to $\ov{\sigma}^2/2$ (see \cite[Theorem 4.5 (iii)]{masi} and \cite[(3.17), (3.19)]{kun}). 

The problem faced is that the function $\phi$ is not practically computable. An idea is to replace $\phi$ by $R_\mu \mathfrak{d}$, where $R_\mu$ is the resolvent operator defined in (\ref{def:Rmu}), and $\mu > 0$ is a small parameter. In the words of \cite{yuri}, the function $R_\mu \mathfrak{d}$ is an ``almost solution'' of the original Poisson equation (\ref{Poissonequation}). As discussed in the introduction of \cite{gloria1}, the function $R_\mu \mathfrak{d}$ can be computed in practice. 

At this point, one expects $A_0(R_\mu \mathfrak{d})$, $A_1(R_\mu \mathfrak{d})$ and $A_2(R_\mu \mathfrak{d})$ to approach $\ov{\sigma}^2/2$ as $\mu$ tends to $0$. These quantities are however no longer equal, and a computation (following \cite[p. 50]{kun}) shows that
$$
A_0(R_\mu \mathfrak{d}) = A_1(R_\mu \mathfrak{d}) +\mu \E[(R_\mu \mathfrak{d})^2] =  A_2(R_\mu \mathfrak{d}) + 2 \mu \E[(R_\mu \mathfrak{d})^2].
$$
In other words, these approximations are of the form
$$
A_k(R_\mu \mathfrak{d}) = A_0(R_\mu \mathfrak{d}) - k \mu \E[(R_\mu \mathfrak{d})^2],
$$
for some $k \in \R$. It turns out that, among the family of possible approximations $(A_k(R_\mu \mathfrak{d}))_{k \in \R}$, all approximations are asymptotically of the same order, except for $k = 2$, for which the approximation is better. \cite{gloria1, gloria2} have indeed chosen $A_2(R_\mu \mathfrak{d})$ as their approximation, while \cite[Theorem 2.1]{yuri} had chosen $A_1(R_\mu \mathfrak{d})$ (but obtained a non-optimal result anyways). A spectral computation gives that
\begin{eqnarray*}
A_k(R_\mu \mathfrak{d}) - \frac{\ov{\sigma}^2}{2}  & = & \cE(\phi,\phi) - \cE(R_\mu \mathfrak{d},R_\mu \mathfrak{d}) - k\mu \E[(R_\mu \mathfrak{d})^2] \\
& = & \int \left(\frac{1}{\lambda} - \frac{\lambda}{(\lambda+\mu)^2} - \frac{k \mu}{(\lambda+\mu)^2} \right) \ \d e_\mathfrak{d}(\lambda) \\
& = & \int \frac{\mu^2 + (2-k) \lambda \mu}{\lambda(\lambda + \mu)^2} \ \d e_\mathfrak{d}(\lambda).
\end{eqnarray*}
For any $f \in L^2(\P)$ satisfying $\int \lambda^{-1} \ \d e_f(\lambda) < +\infty$, we thus define
$$
I_{k,\mu}(f) = \int \frac{\mu^2 + (2-k) \lambda \mu}{\lambda(\lambda + \mu)^2} \ \d e_f(\lambda).
$$ 
Its behaviour as $\mu$ tends to $0$ can be described the following way.
\begin{prop}
\label{erreur2}
Let $f \in L^2(\P)$ and $\alpha > 1$. Under one of the equivalent conditions (\ref{tcl:equiv1}), (\ref{tcl:equiv3}) of Theorem~\ref{tcl}, there exists $C > 0$ such that for any $\mu > 0$~:
$$
0 \le I_{2,\mu}(f) \le 
\left|
\begin{array}{ll}
C \mu^{\alpha - 1} & \text{if }  \alpha < 3,\\
C \mu^2\ln(\mu^{-1}) & \text{if } \alpha = 3 ,\\
C \mu^2 & \text{otherwise.}  
\end{array}
\right.
$$
Reciprocally~:
\begin{equation}
\label{ineqev}
\int_0^\mu \frac{1}{\lambda} \ \d e_f(\lambda) \le 4 I_{2,\mu}(f),
\end{equation}
and, if $f \neq 0$, then there exists $C > 0$ such that, for any $\mu$ small enough,
\begin{equation}
\label{mu2}
C \mu^2 \le I_{2,\mu}(f).
\end{equation}
On the other hand, if $k \in \R \setminus \{2\}$, $f \neq 0$ and $\alpha > 2$, then there exists $C_1,C_2 > 0$ such that, for any $\mu$ small enough,
$$
C_1 \mu \le \left| I_{k,\mu}(f) \right| \le C_2 \mu.
$$
\end{prop}
\begin{proof}
The proof is similar to the proof of Theorem~\ref{tcl}. We decompose $I_{2,\mu}(f)$ as
$$
2 \int_{\lambda = 0}^{+\infty} \int_{\delta = \lambda}^{+\infty} \frac{\mu^2}{(\mu+\delta)^3} \ \d \delta \ \frac{1}{\lambda} \ \d e_f(\lambda). 
$$
By Fubini's theorem, this integral can be rewritten as
$$
2 \int_{\delta = 0}^{+\infty} \frac{\mu^2}{(\mu+\delta)^3} \int_{\lambda = 0}^\delta \frac{1}{\lambda} \ \d e_f(\lambda) \ \d \delta.
$$
Under assumption (\ref{tcl:equiv3}) of Theorem~\ref{tcl}, this integral is bounded by a constant times 
$$
\int_{\delta = 0}^{+\infty} \frac{\mu^2}{(\mu+\delta)^3} \min(\delta^{\alpha - 1},1) \ \d \delta.
$$
The integral obtained when $\delta$ ranges in $[1,+\infty)$ can be computed explicitely, and is smaller than $\mu^2$. For the remaining part, a change of variable leads to
$$
\mu^{\alpha - 1} \int_0^{1/\mu} \frac{1}{(1+u)^3} u^{\alpha - 1} \ \d u,
$$
from which the first part of the proposition follows. Inequality (\ref{ineqev}) is clear if one observes that
$$
\frac{\mu^2}{(\lambda + \mu)^2} \ge \frac{1}{4} \ \1_{[0,\mu]}(\lambda).
$$
For the next claim, let $\delta > 0$ be such that 
\begin{equation}
\label{fneq0}
\int_{0}^{\delta} \d e_f(\lambda) > 0. 
\end{equation}
Then 
$$
I_{2,\mu}(f) \ge \frac{\mu^2}{\delta (\delta+\mu)^2} \int_0^{\delta} \d e_f(\lambda),
$$
which shows (\ref{mu2}). For the last part, let us assume that $k > 2$. One can decompose $I_{k,\mu}(f)$ as
\begin{equation}
\label{decompIkmu}
\left(\int_{0}^{\mu/(k-2)} + \int_{\mu/(k-2)}^{+\infty} \right) \frac{\mu^2 - (k-2) \lambda \mu}{\lambda (\lambda + \mu)^2} \ \d e_f(\lambda).
\end{equation}
In this expression, the integrand is positive on the first interval of integration, and negative on the second one. The first integral is thus positive, and bounded by
$$
\int_{0}^{\mu/(k-2)} \frac{\mu^2}{\lambda (\lambda + \mu)^2} \ \d e_f(\lambda) \le \int_{0}^{\mu/(k-2)} \frac{1}{\lambda } \ \d e_f(\lambda)  \le C \mu^{\alpha - 1},
$$
which is negigible compared to $\mu$ when $\alpha > 2$. The second integral obtained from (\ref{decompIkmu}) can be separated into 
$$
\int_{\mu/(k-2)}^{+\infty} \frac{\mu^2}{\lambda (\lambda + \mu)^2} \ \d e_f(\lambda) - (k-2)\int_{\mu/(k-2)}^{+\infty} \frac{\mu}{(\lambda + \mu)^2} \ \d e_f(\lambda).
$$
The first integral is positive, and bounded by $I_{2,\mu}(f)$, which is negligible compared to $\mu$ when $\alpha > 2$. Let us define $\delta > 0$ such that (\ref{fneq0}) holds. Then
$$
\int_{\mu/(k-2)}^{+\infty} \frac{\mu}{(\lambda + \mu)^2} \ \d e_f(\lambda) \ge \frac{\mu}{(\delta + \mu)^2} \int_{\mu/(k-2)}^{\delta}  \d e_f(\lambda),
$$
which is larger than $C \mu$ for any small enough $\mu$. The proof is similar for $k < 2$ (only simpler). 

Reciprocally, we need to show that, for any $\alpha > 2$, $|I_{k,\mu}(f)| \le C \mu$. Due to the previous observations, it is in fact sufficient to bound 
$$
\int_0^{+\infty} \frac{\mu}{(\lambda + \mu)^2} \ \d e_f(\lambda) = \mu \E[(R_\mu f)^2]. 
$$
We postpone the proof of this last fact to the proof of Proposition~\ref{secondmoment}.
\end{proof}

\cite{gloria2} obtain the following result.
\begin{thm}[\cite{gloria2}]
\label{thm:gloria}
For $\mathfrak{d}$ defined in (\ref{deffung}), and for some $c > 0$, one has~: 
$$
I_{2,\mu}(\mathfrak{d}) \le 
\left|
\begin{array}{ll}
C \mu \ln(\mu^{-1})^c & \text{if } d = 2, \\
C \mu^{3/2} & \text{if } d = 3, \\
C \mu^2 \ln(\mu^{-1}) & \text{if } d = 4, \\
C \mu^2 & \text{if } d \ge 5. \\
\end{array}
\right.
$$
\end{thm}
They also argue that, except possibly for the logarithmic term in dimension $2$, these bounds cannot be improved \cite[Appendix]{gloria1}. This result has the following consequences.

\begin{cor}
\label{cor:gloria}
For any $\eps > 0$, there exists $C > 0$ such that, for any $t > 0$~:
\begin{equation}
\label{cor:eq1}
\E[(\mathfrak{d}_t)^2] \le 
\left|
\begin{array}{ll}
C t^{-2+\eps} & \text{if } d = 2, \\
C t^{-5/2} & \text{if } d = 3, \\
C t^{-3+\eps} & \text{if } d = 4, \\
C t^{-3} & \text{if } d \ge 5, \\
\end{array}
\right.
\end{equation}
and moreover,
\begin{equation}
\label{cor:eq2}
0 \le t^{-1} \ov{\E}\left[(\| X_t \|_2)^2\right] - d \ov{\sigma}^2 \le 
\left|
\begin{array}{ll}
C t^{-1+\eps} & \text{if } d = 2, \\
C t^{-1} & \text{if } d \ge 3.
\end{array}
\right.
\end{equation}
\end{cor}
\begin{proof}
The first claim is obtained from Theorem~\ref{thm:gloria} using inequality (\ref{ineqev}) and Theorem~\ref{tcl}. The second one is a consequence of the first, obtained exactly the same way as we proved Proposition~\ref{prop:tclcor}. 
\end{proof}
\noindent \textbf{Remark.} The result in (\ref{cor:eq1}) improves on our Theorems~\ref{main:d<} and \ref{main:d>} when $d \le 9$ (the bound coincides with the one given by Theorem~\ref{main:d>} when $d = 10$, and is weaker for larger dimensions), while (\ref{cor:eq2}) strengthens our Corollary~\ref{tclcor} when $d \le 8$ (and is equivalent otherwise). 

In terms of practical computations, the expectation that one needs to compute in the formula (\ref{gloriaotto}) is still problematic. \cite{gloria1} propose to replace this expectation by a spatial average, evaluated on a single realisation of the environment. This new approximation has expectation $A_2(R_\mu \mathfrak{d})$, but also has random fluctuations. 

The main purpose of \cite{gloria1} is to estimate the $L^2$ norm of these fluctuations. In order to do so, they show \cite[Proposition 1]{gloria1} that, for $d \ge 3$ and for any $q > 0$, there exists a constant $C$ such that 
\begin{equation}
\label{momentsgloria}
\sup_{\mu > 0} \E\left[ |R_\mu \mathfrak{d}|^q \right] \le C,
\end{equation}
and in dimension $2$, that $\E\left[ |R_\mu \mathfrak{d}|^q \right]$ is bounded by some power of $\ln(\mu^{-1})$. The result concerning the case $q = 2$ can easily be linked with the behaviour of the spectral measure. 
\begin{prop}
\label{secondmoment}
Let $f \in L^2(\P)$ and $\alpha > 1$. Under one of the equivalent conditions (\ref{tcl:equiv1}), (\ref{tcl:equiv3}) of Theorem~\ref{tcl}, there exists $C > 0$ such that for any $\mu > 0$~:
$$
\E\left[(R_\mu f)^2\right] \le 
\left|
\begin{array}{ll}
C \mu^{\alpha - 2} & \text{if }  \alpha < 2,\\
C \ln(\mu^{-1}) & \text{if } \alpha = 2 ,\\
C & \text{otherwise,}  
\end{array}
\right.
$$
and reciprocally~:
$$
\int_0^\mu \frac{1}{\lambda} \ \d e_f(\lambda) \le 8 \mu \ \E\left[(R_\mu f)^2\right].
$$
\end{prop}
\begin{proof}
The spectral representation gives us that
\begin{equation}
\label{rmufspec}
\E\left[(R_\mu f)^2\right] = \int \frac{1}{(\lambda + \mu)^2} \ \d e_f(\lambda). 
\end{equation}
We may rewrite this integral as
$$
\int \frac{\lambda}{(\lambda + \mu)^2} \ \frac{1}{\lambda} \ \d e_f(\lambda) = \int_{\lambda = 0}^{+\infty} \int_{\delta = \lambda}^{+\infty} \frac{\delta - \mu}{(\delta + \mu)^3} \ \d \delta \ \frac{1}{\lambda} \ \d e_f(\lambda).
$$
Using Fubini's theorem, and under condition (\ref{tcl:equiv3}) of Theorem~\ref{tcl}, one obtains that this integral is bounded by a constant times
$$
\int_0^{+ \infty} \frac{\delta - \mu}{(\delta + \mu)^3} \min(\delta^{\alpha - 1},1)  \ \d \delta.
$$
The integral over the interval $[1,+\infty)$ is bounded by
$$
\int_1^{+ \infty} \frac{\delta - \mu}{(\delta + \mu)^3} \ \d \delta = \frac{1}{(1+\mu)^2} \le 1,
$$
One is thus left with the study of 
$$
\int_0^{1} \frac{\delta - \mu}{(\delta + \mu)^3} \delta^{\alpha - 1} \ \d \delta,
$$
which, by a change of variable, becomes
$$
\mu^{\alpha - 2} \int_0^{1/\mu} \frac{u-1}{(u+1)^3} u^{\alpha - 1} \ \d u.
$$
The first part of the proposition then follows. Reciprocally, one can see from (\ref{rmufspec}) that~:
\begin{equation}
\label{comp0}
\int_0^\mu \d e_f(\lambda) \le 4 \mu^2 \E\left[(R_\mu f)^2\right].
\end{equation}
We then note that
$$
\int_0^\mu \frac{1}{\lambda} \ \d e_f(\lambda) = \int_{\lambda = 0}^\mu \int_{\delta = \lambda}^{+\infty} \frac{1}{\delta^2} \ \d e_f(\lambda).
$$
We first bound the part of this double integral for which $\delta$ ranges in $[\mu,+\infty)$~:
$$
\int_{\lambda = 0}^\mu \int_{\delta = \mu}^{+\infty} \frac{1}{\delta^2} \ \d e_f(\lambda) =  \frac{1}{\mu} \int_{\lambda = 0}^\mu \d e_f(\lambda).
$$
This term is, by (\ref{comp0}), bounded by $4 \mu \E[(R_\mu f)^2]$. Using Fubini's theorem, the remaining part of the double integral is equal to
$$
\int_{\delta = 0}^{\mu} \frac{1}{\delta^2} \int_{\lambda = 0}^\delta \d e_f(\lambda) \ \d \delta.
$$
Using the inequality (\ref{comp0}) once again, we obtain that this second term is also bounded by $4 \mu \E[(R_\mu f)^2]$, which finishes the proof of the proposition. 
\end{proof}

Considering Corollary~\ref{cor:gloria}, it seems reasonable to expect the exponent of decay associated with the function $\mathfrak{d}$ to be equal to $d/2 + 1$ in any dimension. A simple argument enables to prove that if one considers the environment seen by the simple random walk, then it is indeed the case. One may indeed write $\mathfrak{d}$ as
$$
\mathfrak{d} = f(\theta_{e_1} \ \omega) - f(\omega),
$$
where $f(\omega) = \omega_{0,-e_1} - \E[\omega_{0,-e_1}]$. For the process of the environment seen by the simple random walk, as we have seen in the proof of Proposition~\ref{contract}, the semi-group and the space translations commute, and thus
\begin{eqnarray*}
\E[(\mathfrak{d}_t^\circ)^2] & = & \E\left[\left(f_t^\circ(\theta_{e_1} \ \omega) - f_t^\circ(\omega) \right)^2 \right] \\
& \le & \cE^\circ(f_t^\circ,f_t^\circ) = \int \lambda e^{-2 \lambda t} \ \d e_f^\circ(\lambda),
\end{eqnarray*}
where $e_f^\circ$ is the spectral measure of $- \L^\circ$ projected on the function $f$. With this representation, and knowing that the exponent of decay to equilibrium of $f$ is $d/2$ (a fact which follows from Theorem~\ref{main:indep}, assuming that the conductances are square integrable), one obtains that the exponent of decay to equilibrium of $\mathfrak{d}$ is $d/2 + 1$ by following the proof of Theorem~\ref{tcl}.



\noindent \textbf{Acknowledgments.} I would like to thank my Ph.D. advisors, Pierre Mathieu and Alejandro Ram\'irez, for many insightful discussions about this work as well as detailed comments on earlier drafts. I also thank an anonymous referee for suggesting possible links between the works of \cite{yuri,gloria1,gloria2} and the present paper.

\end{document}